\documentclass[10 pt, reqno, twoside]{amsart}

\usepackage[latin1]{inputenc}
\usepackage{lmodern}
\usepackage[T1]{fontenc}
\usepackage[english]{babel}
\usepackage{ngerman}
\usepackage{amsmath}
\usepackage{amssymb}
\usepackage{amsfonts}
\usepackage{amstext}
\usepackage{amsthm}
\usepackage{hyperref}
\usepackage{xcolor}
\usepackage{graphicx}
\usepackage{subfigure}

\def\D{D}
\def\H{H}

\usepackage[a4paper, hdivide={4.2 cm,,4.2 cm},vdivide={4.9cm,,4.2cm}]{geometry}
\begin{document}

\theoremstyle{plain}
	\newtheorem{Pp}{Proposition}[section]
	\newtheorem{Thm}[Pp]{Theorem}
	\newtheorem{Lm}[Pp]{Lemma}
	\newtheorem{Cor}[Pp]{Corollary}
	\newtheorem*{H1}{Assumption (H1)}
	\newtheorem*{H2}{Assumption (H2)}
	\newtheorem*{H3}{Assumption (H3)}
	\newtheorem*{H4}{Assumption (H4)}
\theoremstyle{definition}
  \newtheorem*{Data}{Data (D)}
	\newtheorem{Df}[Pp]{Definition}
	\newtheorem{ex}[Pp]{Example}
	\newtheorem{Cond}[Pp]{Condition}
	\newtheorem{Ass}[Pp]{Assumption}
	\newtheorem{Rm}[Pp]{Remark}
	\newtheorem{Int}[Pp]{Motivation and Interpretation}

\title[Hypocoercivity for Kolmogorov backward equations]{Hypocoercivity for Kolmogorov backward evolution equations and applications}

\author{Martin Grothaus}
\address{Martin Grothaus, Mathematics Department, University of Kaiserslautern, \newline 
P.O.Box 3049, 67653 Kaiserslautern, Germany. {\rm \texttt{Email:~grothaus@mathematik.uni-kl.de}},\newline
Functional Analysis and Stochastic Analysis Group, \newline
{\rm \texttt{URL:~http://www.mathematik.uni-kl.de/fuana/ }}} 

\author{Patrik Stilgenbauer}
\address{
Patrik Stilgenbauer, Mathematics Department, University of Kaiserslautern, \newline
P.O.Box 3049, 67653 Kaiserslautern, Germany. 
{\rm \texttt{Email:~stilgenb@mathematik.uni-kl.de}}, \newline
Functional Analysis and Stochastic Analysis Group, \newline
{\rm \texttt{URL:~http://www.mathematik.uni-kl.de/fuana/ }}}

\date{\today}



\begin{abstract}
In this article we extend the modern, powerful and simple abstract Hilbert space strategy for proving hypocoercivity that has been developed originally by Dolbeault, Mouhot and Schmeiser in \cite{DMS10}. As well-known, hypocoercivity methods imply an exponential decay to equilibrium with explicit computable rate of convergence. Our extension is now made for studying the long-time behavior of some strongly continuous semigroup generated by a (degenerate) Kolmogorov backward operator $L$. Additionally, we introduce several domain issues into the framework. Necessary conditions for proving hypocoercivity need then only to be verified on some fixed operator core of $L$. Furthermore, the setting is also suitable for covering existence and construction problems as required in many applications. The methods are applicable to various, different, Kolmogorov backward evolution problems. As a main part, we apply the extended framework to the (degenerate) spherical velocity Langevin equation. The latter can be seen as some kind of an analogue to the classical Langevin equation in case spherical velocities are required. This model is of important industrial relevance and describes the fiber lay-down in the production process of nonwovens. For the construction of the strongly continuous contraction semigroup we make use of modern hypoellipticity tools and pertubation theory.
\end{abstract}

\maketitle

\section{Introduction} \label{section_Introduction}

In an interesting recently appeared research article Dolbeault, Mouhot and Schmeiser developed a simple strategy for proving exponential decay to zero of specific strongly continuous semigroups with associated (e.g.~Fokker-Planck) generator denoted by $L$ in a Hilbert space framework. Here $L$ must be of the form $L=S-A$ with $S$ symmetric and $A$ antisymmetric such that $S$ and $A$ are interacting in a crucial way and satisfy certain coercivity assumptions and boundedness relations, see \cite[Sec.~1.3]{DMS10} and its previous article \cite{DMS09}. It was the great idea and one of the main achievements of Dolbeault, Mouhot and Schmeiser to find a suitable entropy functional in terms of the considered operators, which is equivalent to the underlying Hilbert space norm, for measuring the exponential convergence to equilibrium, see \cite[Sec.~1.3]{DMS10}. Neglecting any domain issues and questions concerning the construction of the semigroup first, the four necessary conditions introduced by Dolbeault, Mouhot and Schmeiser can then in principal easily be checked in applications and imply the desired ergodicity behavior of the semigroup. As consequence, a wide range of applications can be studied and discussed with the help of these new tools. The authors strategy may be called hypocoercivity, see Villani's memoirs (\cite{Vil09} and \cite{Vil06}) for explanation and some fundamental studies concerning this new mathematical research area. Besides \cite{Vil09} and \cite{DMS10}, there is a huge list of references dealing with results related to hypocoercivity, hypoellipticity and the (exponential) long-time behavior via using analytic approaches. The interested reader may consult also \cite{AMTU01}, \cite{Cal12}, \cite{Dua11}, \cite{DEVI01}, \cite{DEVI05}, \cite{HeNi06}, \cite{HeNi04},  \cite{HN05}, \cite{LLS07}, \cite{MN06}, \cite{SG06} and \cite{OP11} where we have mentioned only a few of these references. We remark, that in all these articles mainly Fokker-Planck type equations are regarded.

In the underlying article, we consider and extend the hypocoercivity strategy from \cite{DMS10} described previously. Our extension now concerns Kolmogorov backward type evolution equations which are again formulated as an abstract Cauchy problem in a Hilbert space framework. Additionally, we include domain issues and extend the setting from \cite{DMS10} to the case in which we know an operator core for the (Kolmogorov backward) generator $L$ of some strongly continuous semigroup of interest. Conditions (H1)-(H4) from \cite[Sec.~1.3]{DMS10} are adapted in a suitable way and are now mainly formulated on the fixed operator core of $L$. We supplement our framework by developing sufficient criteria implying the necessary assumptions. Whereas the general setting from \cite{DMS10} applies to Hilbert spaces of the form
\begin{align*}
\H=\Big\{ f \in L^2(E,\mu) ~\Big|~ \int_E f \, \mathrm{d}\mu=0 \Big\}
\end{align*}
with $\mu$ the desired equilibrium (probability) distribution on some suitable state space $E$, our framework is formulated on $L^2(E,\mu)$. Together with some data conditions (D), this setup can simplify domain issues and is suitable for showing essential m-dissipativity of $L$ and essential selfadjointness of a related operator occuring in the framework. We emphasize that several Kolmogorov backward hypocoercivity problems can be studied within this setup, in a similar way as we do in our application.

Moreover, our extended framework can be formulated in similar form in some Fokker-Planck setting (again with altered underlying Hilbert space structure compared to \cite{DMS10} and in terms of operator cores). However, due to our application, we are interested in studying the long-time behavior of some specific Kolmogorov backward evolution equations. So we pay attention only on the backward formulation in this article.

The discussion and derivation of the Kolmogorov backward setting is done in Section \ref{Section_Extension_Hypocoercivity}, see especially \textbf{Data (D)}, \textbf{ Assumptions (H1)-(H4)} and \textbf{Theorem \ref{Thm_Hypocoercivity}} as well as Corollary \ref{Cor_H3_equivalent}, Lemma \ref{Pp_suff_H4_part1} and Proposition \ref{Pp_suff_H4_part2} therein and compare the conditions with the original ones from \cite[Sec.~1.3]{DMS10}. 

Afterwards, see Section \ref{Hypocoercivity_section_Spherical_velocity_Langevin}, we apply the previously developed framework to study the ergodic behavior of the spherical velocity Langevin type process which is prescribed by the manifold-valued Stratonovich stochastic differential equation with state space $\mathbb{M}=\mathbb{R}^d \times \mathbb{S}$ of the form
\begin{align} \label{Fiber_Model_Intro}
&\mathrm{d}x_t = \omega_t \, \mathrm{dt} \\
&\mathrm{d}\omega_t = - \frac{1}{d-1}(I- \omega_t \otimes \omega_t) \,\nabla V (x_t) \, \mathrm{dt} + \sigma\,(I-\omega_t \otimes \omega_t) \circ \mathrm{d}W_t. \nonumber
\end{align}
whose associated Kolmogorov backward operator $L$ reads
\begin{align} \label{Fiber_Operator_Intro}
L= \omega \cdot \nabla_x - \text{grad}_{\mathbb{S}} \Phi \cdot \nabla_\omega + \frac{1}{2} \sigma^2 \, \Delta_{\mathbb{S}}\,~\mbox{ with }\,~\Phi(x,\omega)= \frac{1}{d-1} \,\nabla_x V(x) \cdot \omega.
\end{align}
For details on \eqref{Fiber_Model_Intro}, see \cite{GKMS12} or \cite{GS12} (and the related articles \cite{KMW12}, \cite{GKMW07}). Here  $d \in \mathbb{N}$ with $d \geq 2$. $W$ is a standard $d$-dimensional Brownian motion, $z \otimes y = z y^T$ for $z,y \in \mathbb{R}^{d}$ and $y^T$ is the transpose of $y$. $\mathbb{S}=\mathbb{S}^{d-1}$ denotes the unit sphere with respect to the euclidean norm $|\cdot|$ in $\mathbb{R}^d$, $\text{grad}_{\mathbb{S}} \,\psi \cdot \nabla_\omega$ or $\text{grad}_{\mathbb{S}} \,\psi$ the spherical gradient of some $\psi \in C^\infty(\mathbb{S})$ and $\Delta_{\mathbb{S}}$ the Laplace Beltrami on $\mathbb{S}$. $x$ always denotes the space variable in $\mathbb{R}^d$ and $\omega$ the velocity component in $\mathbb{S} \subset \mathbb{R}^d$ and all vectors in euclidean space are understood as column vectors. The potential $V:\mathbb{R}^d \rightarrow \mathbb{R}$ is specified later on and $\sigma$ is a finite constant with $\sigma \geq 0$. $\left(\cdot,\cdot\right)_{\text{euc}}$ or $\cdot$ denotes the euclidean scalar product and $\nabla$, or $\nabla_x$ respectively, the usual gradient in $\mathbb{R}^d$. We also refer to Section \ref{Section_Formulas_Notations}, where the geometric language is made precise.

The geometry of the spherical velocity Langevin process \eqref{Fiber_Model_Intro} is explained in \cite{GS12}. Indeed, the equation can informally be derived from the classical Langevin equation (more precisely, this is (1.1) in \cite{GS12}) via projecting an infinitesimal step of the velocity component in the classical Langevin model onto the sphere. The resulting model equation reduces then to \eqref{Fiber_Model_Intro}. Hence \eqref{Fiber_Model_Intro} describes the evolution of a particle moving under the influence of an external forcing field with velocity of euclidean norm equal to $1$ and whose velocity components are stochastically pertubated by a (spherical) Brownian motion. Of particular industrial interest is its two-dimensional ($d=2$) and three-dimensional ($d=3$) version which exactly arises in industrial applications as the so-called fiber lay-down process in modeling virtual nonwoven webs, see \cite{KMW12} and \cite{GKMS12} and other references therein. Here the restriction that the velocity component lives on the sphere can also be interpreted that the fiber lay-down curve is just arc-length parametrized. Especially, the (exponential) decay to equilibrium of this process is strongly connected to the fleece structure of the resulting nonwoven object. So from an industrial point of view, studying and analyzing the ergodic behavior is therefore an indispensable task which is of course also interesting when considered the latter as a purely mathematical problem. For further motivation on this, see again \cite{GKMS12}. Moreover, models similar to \eqref{Fiber_Model_Intro}, i.e., spherical velocity or arc-length parametrized stochastic models are appearing also in some other mathematical research areas besides the fiber lay-down application. Here we mention \cite{CKMT10} in which an extended two-dimensional version of \eqref{Fiber_Model_Intro} is used for modeling the swarming behavior. Furthermore, see \cite{Mum90}, where the two-dimensional operator \eqref{Fiber_Operator_Intro} with velocity component parametrized in terms of an angle and with potential equal to zero is introduced in computer vision. Finally, as pointed out to us by Isma\"{e}l Bailleul, a similar spherical velocity model close to \eqref{Fiber_Model_Intro} arises in a relativistic framework in modeling random dynamics on Lorentzian manifolds, see \cite{FJ07} and \cite{Bai10}. 

The long-time behavior of the two-dimensional version of \eqref{Fiber_Model_Intro} is already investigated in several research articles. Mathematical demanding problems arising since the equation and its associated Kolmogorov generator are degenerate. We remark that the two-dimensional model can equivalently be formulated on $\mathbb{R}^2 \times \mathbb{R} / {2 \pi \mathbb{Z}}$, thus, the usage of a differential geometric language can in principal be avoided in this case. Let us describe the results obtained so far for the latter model. By using the theory of (generalized) Dirichlet forms and operator semigroups, an ergodic theorem with explicit computable rate of convergence has been presented in \cite{GK08}. Furthermore, Dolbeault, Klar, Mouhot and Schmeiser applied the hypocoercivity methods from \cite{DMS10} in order to get an exponential rate of convergence for the two-dimensional fiber lay-down equation. This hypocoercivity approach is used and generalized in \cite{GKMS12} to the $d$-dimensional model given by \eqref{Fiber_Model_Intro}. In the latter article, a further stochastic strong mixing result under weak potential conditions is derived for the general model \eqref{Fiber_Model_Intro}. However, both of these hypocoercivity approaches for the fiber lay-down equation are done in an algebraic way and do not discuss any domain issues, in particular, a rigorous elaboration of them has been left open, see Remark 6.1 and Remark 6.5 in \cite{GKMS12}. Moreover, let us also mention the articles \cite{KSW11} and \cite{KSW12} which are based on purely stochastic tools such as using Lyapunov type arguments. Therein Kolb, Savov and W\"{u}bker derive strong mixing properties under weak assumptions on the potential and show geometric ergodicity of the two-dimensional fiber lay-down equation (with an included moving belt) under some stronger conditions on the potential.

A mathematical complete elaboration of the desired hypocoercivity theorem for the spherical velocity Langevin process can now precisely be discussed with the help of the abstract Kolmogorov backward setting developed in Section \ref{Section_Extension_Hypocoercivity}. Let us already state the final theorem. Therein $C_c^\infty(\mathbb{X})$ denotes the set of all infinitely often compactly supported differentiable functions on some smooth manifold $\mathbb{X}$ and $C^2(\mathbb{X})$ means twice continuously differentiable on $\mathbb{X}$. Moreover, $\nabla_x^2$ (or $\nabla^2$) denotes the Hessian matrix in $\mathbb{R}^d$. We remark that potentials of the form $V(x)=\sum_{i=1}^d a_i \,x_i^2$, $a_i \geq 0$, which are relevant for the fiber lay-down application, satisfy (after normalization) the necessary conditions below.

\begin{Thm} \label{Hypocoercivity_theorem_spherical_velocity_Langevin}
Let $d \in \mathbb{N}$, $d \geq 2$, and let $\sigma > 0$. We assume that the potential $V:\mathbb{R}^d \rightarrow \mathbb{R}$ is bounded from below, satisfies $V \in C^{2}(\mathbb{R}^d)$ and that $e^{-V} \mathrm{d}x$ is a probability measure on $(\mathbb{R}^d,\mathcal{B}(\mathbb{R}^d))$. Moreover, the measure $e^{-V}\mathrm{d}x$ is assumed to satisfy a Poincar\'e inequality of the form
\begin{align*}
\big\|\nabla f \big\|^2_{L^2(e^{-V}\mathrm{d}x)} \geq \Lambda  \, \left\|\, f - \int_{\mathbb{R}^d} f \, e^{-V}\mathrm{d}x \,\right\|^2_{L^2(e^{-V}\mathrm{d}x)}
\end{align*}
for some $\Lambda \in (0,\infty)$ and all $f \in C_c^\infty(\mathbb{R}^d)$. Furthermore, assume that there exists some constant $c < \infty$ such that
\begin{align*}
\left| \nabla^2 V (x) \right| \leq c \left( 1+ \left| \nabla V(x) \right|\right) \mbox{ for all } x \in \mathbb{R}^d.
\end{align*}
Define $\mu = e^{-V} \mathrm{d}x \otimes \nu$ where $\nu$ is the normalized spherical surface measure. Then the spherical velocity Langevin operator $(L,C_c^\infty(\mathbb{M}))$ is closable on $L^2(\mathbb{M},\mu)$ and its closure $(L,D(L))$ generates a strongly continuous contraction semigroup denoted by $(T(t))_{t \geq 0}$.  Finally, there exists strictly positive constants $\kappa_1,\kappa_2 < \infty$  which are computable in terms of $\Lambda$, $c$, $d$ and $\sigma$ such that for each $g \in {L^2(\mathbb{M},\mu)}$ we have
\begin{align*}
\left\|\,T(t)g - \int_\mathbb{M} g \, \mathrm{d}\mu \,\right\|_{L^2(\mathbb{M},\mu)} \leq \kappa_1 e^{-\kappa_2 \,t}  \left\|\,g - \int_\mathbb{M} g \, \mathrm{d}\mu \,\right\|_{L^2(\mathbb{M},\mu)} ~\mbox{ for all } t \geq 0.
\end{align*}
\end{Thm}

We further remark that the semigroup can be constructed even under very general conditions on $V$, namely $V$ has to be bounded from below and locally Lipschitz continuous, see Theorem \ref{Thm_essential_mdissipativity_sphericalvelocityLangevin_locLipschitz}. Here we apply modern arguments, based on hypoellipticity, from \cite{HN05} as well as use pertubation theory techniques as originally developed in \cite{CG08} and \cite{CG10}. Finally, we refer to \cite{DKMS11} (or also \cite{GKMS12}) concerning the computation of the rate of convergence in terms of $\sigma$.  

Moreover, due to the nature of the equation from above one may think at first sight that methods dealing with the long-time behavior of the classical Langevin equation (or its dual respectively, the so-called linear kinetic Fokker-Planck equation) can also be applied to the spherical velocity Langevin equation in a similar way. However, this is not true at all. It is interesting to note that Villani's Hilbert space hypocoercivity method, see \cite[Theo.~24]{Vil09}, seems to produce non-ending commutator relations for the vector fields occurring in the decomposition of $L$ (as required in \cite{Vil09}) and it is an open problem to apply the methods from \cite{Vil09} to the spherical velocity Langevin equation.

Altogether, the main results obtained in this paper are summarized as follows:
\begin{itemize}
\item
Extending the Hilbert space hypocoercivity method of Dolbeault, Mouhot and Schmeiser from \cite{DMS10} to the Kolmogorov backward setting, see Section \ref{Section_Extension_Hypocoercivity}.
\item
Introducing additionally operator cores into the framework and putting emphasis on domain issues. Thus we now even have the ability to cover existence and construction problems as required in many applications.
\item
Adding a set of data conditions (D) to the conditions (H1)-(H4) from \cite{DMS10}. So we get a complete existence and hypocoercivity Kolmogorov backward setting, see Theorem \ref{Thm_Hypocoercivity}. Moreover, we supplement the extended framework by developing sufficient criteria for the required assumptions. Necessary conditions need now only to be verified on some fixed operator core.
\item
Proving essential m-dissipativity of the spherical velocity Langevin operator with predomain $C^\infty_c(\mathbb{M})$ under weak conditions on the potential, see Section \ref{Section_Construction_Semigroup}, in particular, Theorem \ref{Thm_Essential_mdissipativity_smoothcase} and Theorem \ref{Thm_essential_mdissipativity_sphericalvelocityLangevin_locLipschitz} therein.
\item
Giving a mathematical complete elaboration of the hypocoercivity theorem for the spherical velocity Langevin process. This implies its exponential decay to equilibrium, see Section \ref{Hypocoercivity_section_Spherical_velocity_Langevin} and Theorem \ref{Hypocoercivity_theorem_spherical_velocity_Langevin}.
\end{itemize} 
\section{Extension of the hypocoercivity method of Dolbeault, Mouhot and Schmeiser to the Kolmogorov backward setting} \label{Section_Extension_Hypocoercivity}

As described in the introduction, in this section we extend the hypocoercivity method of Dolbeault, Mouhot and Schmeiser from \cite[Sec.~1.3]{DMS10} to the Kolmogorov backward setting (the notation is explained later on) by putting emphasis on domain issues. We remark again that the elaborations and concepts discussed in the underlying section are motivated and based on the originally ones developed in \cite[Sec.~1.3]{DMS10} (or \cite{DMS09} respectively). Henceforth, the powerful ideas go back to Dolbeault, Mouhot and Schmeiser and some proofs are similar and analogous to the proofs done in \cite[Sec.~1.3]{DMS10}. See especially the proof of Lemma \ref{Boundedness_of_B} and the proof of Theorem \ref{Thm_Hypocoercivity} and compare the formulation of the conditions (H1)-(H4) below with the original conditions (H1)-(H4) as introduced in \cite{DMS10}. But before starting we need some lemmas. For the background in the theory of operator semigroups, the reader may find all informations in \cite{Paz83} and \cite{Gol85}.

\subsection{Some auxiliary lemmas} \label{section_auxiliary_lemmata}

In this section, $\H$ always denotes some real Hilbert space with scalar product $(\cdot,\cdot)_\H$ and induced norm $\| \cdot \|$. All considered operators are assumed to be linear, defined on linear subspaces of $\H$. An operator $(T,D(T))$ with domain $D(T)$ is also abbreviated by $T$. As usual $(T,D(T))$ is called bounded if there exists some constant $c < \infty$ such that $\| Tf \| \leq c \|f\| $ for all $f \in D(T)$. The range of some operator $(T,D(T))$ is abbreviated by $\mathcal{R}(T)$. First we shall recall some basic facts concerning closed operators summarized in the following lemma. 

\begin{Lm} \label{Lm_basics_closed_operators}
Let $(T,D(T))$ be a densely defined, linear operator on $\H$. Let $L$ be a bounded operator with domain $\H$.
\begin{itemize}
\item[(i)]
The adjoint operator $(T^*,D(T^*))$ exists and is closed. If $D(T^*)$ is dense in $\H$, then $(T,D(T))$ is closable and for the closure $(\overline{T},D(\overline{T}))$ it holds $\overline{T}=T^{**}$.
\item[(ii)]
$L^*$ is bounded and $\|L\|=\|L^*\|$.
\item[(iii)]
If $(T,D(T))$ is closed, then $D(T^*)$ is automatically dense in $\H$. Consequently, $T=T^{**}$.
\item[(iv)]
Let $(T,D(T))$ be closed. Then the operator $TL$ with domain 
\begin{align*}
D(TL)=\{ f \in \H~|~Lf \in D(T)\}
\end{align*}
is also closed.
\item[(v)]
$LT$ with domain $D(T)$ need not to be closed. However, $(LT)^*=T^*L^*$.
\end{itemize}
\end{Lm}

Here is the first lemma which will be essential for the upcoming considerations.

\begin{Lm} \label{Lm_APi}
Let $(T,D(T))$ be (anti-)symmetric and let $P:\H \rightarrow \H$ be an orthogonal projection (i.e., $P$ is symmetric and $P^2=P$) satisfying $P(\D) \subset D(T)$ for some subspace $\D \subset D(T)$ which is dense in $\H$. Then we have:
\begin{itemize}
\item[(i)]
$D(T) \subset D((TP)^*)$ and $(TP)^*_{|D(T)}=-P T$ holds in case $T$ is antisymmetric and $(TP)^*_{|D(T)}=P T$ in case $T$ is symmetric.
\item[(ii)]
$P (T P)^*= (T P)^*$ on $D((T P)^*)$.
\end{itemize}
\end{Lm}

\begin{proof} Assume that $T$ is antisymmetric. The symmetric case is proven analogously. Part (i): For $f \in D(TP)$ and $g \in D(T)$ we have
\begin{align*}
\left( T P f, g \right)_{\H} = - \left( P f, T g \right)_{\H} = - \left(  f, P T g \right)_{\H}
\end{align*}
since $P f \in D(T)$ and $(T,D(T))$ is antisymmetric. Thus $g \in D((TP)^*)$. Furthermore, $(TP)^*g=-P Tg$ since $D(TP)$ is dense in $\H$. For part (ii) let $f \in D((TP)^*)$ and $g \in \D$. Then 
\begin{align*}
\left( P (TP)^*f,g \right)_{\H} = \left(  (TP)^*f, P g \right)_{\H} = \left(  f, (TP )P g \right)_{\H} = \left(  f, TP  g \right)_{\H} = \left(  (T P)^* f,  g \right)_{\H}
\end{align*}
where the third equality follows since $P(P g)=P g \in D(T)$, i.e, $P g \in D(T P)$ and the last one holds since $g \in D(TP)$. The claim follows since $\D$ is dense in $\H$.
\end{proof}

Assume that $(T,D(T))$ is closed and let the conditions from Lemma \ref{Lm_APi} hold. Then $(T P, D(T P))$ is also closed and densely defined. Thus von Neumann's theorem, see e.g.~\cite[Theo.~5.1.9]{Ped89}, implies that the operator $I+(TP)^*TP$ with domain 
\begin{align*}
D((TP)^*(TP))=\big\{f \in D(TP)~\big|~TP f \in D((TP)^*)\,\big\}
\end{align*}
is injective and surjective (with range equal to $H$) and admits a bounded linear inverse. We define the operator $B$ with domain $D((TP)^*)$ via
\begin{align} \label{definition_of_B}
B:=(I+(TP)^*T P)^{-1}(TP)^*.
\end{align}
By using Lemma \ref{Lm_basics_closed_operators}, it is easily checked that $B^*=TP(I+(TP)^*T P)^{-1}$. But $B^*$ is closed and has domain $D(B^*)=\H$, hence $B^*$ is bounded. Consequently, see Lemma \ref{Lm_basics_closed_operators}, $(B,D((TP)^*))$ is already bounded. Its unique extension to some continuous linear operator on $\H$ is denoted by $B$. 

\begin{Rm}
One can even show that the operator $B$ is explicitly given on $\H$ as 
\begin{align*} 
B=(TP)^*(I+TP(TP)^*)^{-1},
\end{align*}
see \cite[Theo.~5.1.9]{Ped89}. And moreover, the latter has operator norm less or equal than $1$. However, we do not need this information in the following.
\end{Rm}

Under some additional assumptions, the operator norm of $B$ can further be specified. Compare the upcoming statement with its original version, see Lemma 1 in \cite{DMS10}.

\begin{Lm} \label{Boundedness_of_B}
Let $(T,D(T))$ be closed and (anti-)symmetric and let $P:\H \rightarrow \H$ be an orthogonal projection satisfying $P(\D) \subset D(T)$ for some subspace $\D \subset D(T)$ which is dense in $\H$. Then $P B =B$. Moreover, additionally assume that
\begin{align} \label{Assumption_(A1)}
P T P _{| \D}=0.
\end{align}
Then we even have
\begin{align} \label{norm_of_B}
\|Bf\| \leq \frac{1}{2} \|(I-P)f\| ~~\mbox{ for all } f \in \H.
\end{align}
Furthermore, the operator $(TB,\D)$ is bounded and it holds 
\begin{align*}
\|TBf\| \leq \| (I- P)f\| ~~\mbox{ for all } f \in \D.
\end{align*}
\end{Lm}

\begin{proof}
Let $f \in D((TP)^*)$ and define $g$ as $g=Bf$. Thus $g \in D((TP)^*(TP))$ and $g \in D(TP)$ by \eqref{definition_of_B}. So by the representation of $B$ on $D((TP)^*)$ it holds
\begin{align*} 
(T P)^*f=g + (TP)^*(TP)g.
\end{align*}
Applying $P$ on both sides and using Lemma \ref{Lm_APi}, we conclude $Pg=g$. Hence $PB=B$ holds on $\H$ due to continuity of $B$ and $P$. By taking the scalar product with respect to $g$ on both sides of the latter equation  we get
\begin{align} \label{Eq_Boundedness_of_B_2}
\|g\|^2 + \| T P g\|^2 = \|g\|^2 + \left( (TP)^*(TP)g, g \right)_\H = \left((T P)^*f,g \right)_\H. 
\end{align}
Now let $f \in \D$. Then $P f \in D(T)$ and $(TP)^*P f =0$ by Lemma \ref{Lm_APi} and Assumption \eqref{Assumption_(A1)}. By using additionally that $(I-P)f \in D(T) \subset D((T P)^*)$, Equation $\eqref{Eq_Boundedness_of_B_2}$ yields
\begin{align} \label{Eq_Boundedness_of_B_3}
\|g\|^2 + \| T P g\|^2 &=  \left((T P)^*(I-P)f,g \right)_\H  = \left((I-P)f,TP g \right)_\H \nonumber\\
&\leq \| (I-P)f\| \| T P g\|  \leq \frac{1}{4} \| (I-P)f\|^2 + \| T P g\|^2 .
\end{align}
Hence for all $f \in \D$ we get $\|Bf\| \leq \frac{1}{2} \| (I-P)f \|$. Now \eqref{norm_of_B} follows since $\D$ is dense in $\H$. For the last statement, observe that really $\D \subset D((TP)^*) \subset D(TB)$ since $B(D((TP)^*)) \subset D(T P)$ and $P B=B$. Then use again Inequality \eqref{Eq_Boundedness_of_B_3}.
\end{proof}

\subsection{The hypopcoercivity method in the Kolmogorov backward setting} Now we start with the previously mentioned extension of the original hypocoercivity method from \cite{DMS10} to the Kolmogorov backward setting. First some heuristics and physical motivation to understand the framework. The motivation shall be read, of course, only on an informal level.
\medskip

\textbf{Motivation and Interpretation.} First let us consider a manifold-valued (Stratonovich) stochastic differential equation (SDE) with state space $E$ being some Riemannian manifold $\mathbb{M}$. Let $X^x=(X^x_t)_{t \geq 0}$ denotes the solution to this SDE starting from $x \in \mathbb{M}$. Then $X^x$ provides a solution to the $L$-martingale problem where $L$ is the so-called Kolmogorov backward operator associated to the SDE. Let $g:\mathbb{M} \rightarrow \mathbb{R}$ be a suitable test function. Then $u(t,x)=\mathbb{E}[g(X^x_t)]$ solves the Kolmogorov backward partial differential equation (PDE) of the form $\partial_t u (t,x) = L u (t,x)$ in the classical sense (pointwisely). Here $L$ only acts on the $x$-variable and $\mathbb{E}$ denotes expectation. Often one is interested in studying the longtime behavior of $X$. If $\mu$ denotes the explicitly known candidate for the stationary (normalized) distribution, whose density with respect to the volume measure on $\mathbb{M}$ is a stationary solution to the associated Fokker-Planck equation, then $X_t^x$ shall be distributed accordingly to $\mu$ for large values of $t \geq 0$. In other words, one is interested in studying the convergence of $u(t,x)$ to $\int_{\mathbb{M}} g\, \mathrm{d}\mu$ as $t \rightarrow \infty$. 

Next we consider the Kolmogorov-backward PDE as an abstract Cauchy problem  $\dot{u}(t)  = Lu(t)$ on the Hilbert space $\H=L^2(E,\mu)$ with $\mu$ being the desired equilibrium distribution. Now $u(t)=T(t)g$ where $(T(t))_{t \geq 0}$ is the strongly continuous semigroup which $L$ assumes to generate on $\H$ and $E$ may even be a more general state space. This is indeed a natural choice for the underlying Hilbert space. Therefore, we just remark that under suitable assumptions on the diffusion operator $L$, the theory of (generalized) Dirichlet forms then really implies the existence of a Markov-process $X$ solving the $L$-martingale problem such that $T(t)f(x)= \mathbb{E}^x[g(X_t)]$ holds for quasi-every $x \in E$, for instance see \cite{Roe99}, \cite{Tru00}, \cite{St99} and \cite{CG08} as well as conditions (D6) and (D7) below. Motivated by the previous considerations, we are interested in studying the convergence, especially the exponential decay, of $T(t)g$ to $\left(g,1\right)_\H$ in $\H$ as $t \rightarrow \infty$. In case $L$ is degenerate, mathematical demanding problems arising. The whole program is realized in the upcoming setup, called the \textit{Kolmogorov backward (hypocoercivity) setting}.
\medskip

After this motivation, we now switch to a precise, formal, general framework suitable for studying the long-time behavior of solutions to the Cauchy problem of some Kolmogorov backward PDE using a Hilbert space approach.

\begin{Data} \label{Data} We require the following conditions which are assumed for the rest of this section without further mention them again.
\begin{itemize}
\item [(D1)] \textit{The Hilbert space:} Let $(E,\mathcal{F},\mu)$ be some probability space and define $\H$ to be $\H=L^2(E,\mu)$ equipped with the usual standard scalar product $\left(\cdot,\cdot\right)_\H$.
\item [(D2)] \textit{The $C_0$-semigroup and its generator $L$:} $(L,D(L))$ is some linear operator on $\H$ generating a strongly continuous semigroup (abbreviated $C_0$-semigroup) $(T(t))_{t \geq 0}$.
\item [(D3)] \textit{Core property of $L$:}  Let $\D \subset D(L)$ be a dense subspace of $\H$ which is a core for $(L,D(L))$. 
\item [(D4)] \textit{Decomposition of $L$:} Let $(S,D(S))$ be symmetric and let $(A,D(A))$ be closed and antisymmetric on $\H$ such that $\D \subset D(S) \cap D(A)$ as well as $L_{|\D}=S-A$.
\item [(D5)] \textit{Orthogonal projections:} Let $P:\H \rightarrow \H$ be an orthogonal projection satisfying $P(\H) \subset D(S)$, $S P=0$ as well as $P(\D) \subset D(A)$, $AP (\D) \subset D(A)$. Moreover, we introduce
$P_{S}:\H \rightarrow \H$ as
\begin{align*}
P_S f:= P f + \left( f,1 \right)_\H,~f \in \H.
\end{align*}
\item [(D6)] \textit{The invariant measure:} Let $\mu$ be invariant for $(L,\D)$ in the sense that
\begin{align} \label{Df_invariant_measure}
\left( Lf,1 \right)_\H = \int_E Lf\, \mathrm{d}\mu=0 ~\mbox{ for all $f \in \D$}.
\end{align}
\item [(D7)] \textit{Conservativity:}  $1 \in D(L)$ and $L1=0$.
\end{itemize}
\end{Data}

\begin{Rm}
Informally speaking, the reader should think of $P_S$ as being the projection onto the kernel of $S$ (provided it exists). This is motivated by Condition (H2) below and see also the choice of $P_S$ in our application in Section \ref{Hypocoercivity_section_Spherical_velocity_Langevin}.
\end{Rm}

The first basic assumption is introduced next. Compare it with the original framework from \cite{DMS10}.

\begin{H1} \label{H1} Assume that
\begin{align*} 
P A P _{\,| \D}=0.
\end{align*}
\end{H1}

\begin{Df} First the operator $B$ is defined similar as in Subsection \ref{section_auxiliary_lemmata} as the unique extension of $(B,D((AP)^*))$ to some continuous linear operator on $\H$ where
\begin{align*}
B:=(I+(AP)^*A P)^{-1}(AP)^* \mbox{ on } D((AP)^*). 
\end{align*}
Let $0 \leq \varepsilon < 1$. The \textit{modified entropy functional} is defined as in \cite{DMS10} by
\begin{align*}
\mathrm{H}_{\varepsilon}[f]=\frac{1}{2} \|f\|^2 + \varepsilon \left(Bf,f\right)_\H,~f \in \H.
\end{align*}
Now assume that (H1) holds. Then Lemma \ref{Boundedness_of_B} yields
\begin{align} \label{equivalence_H_norm}
\frac{1-\varepsilon}{2} \|f\|^2 \leq \mathrm{H}_\varepsilon[f] \leq \frac{1+\varepsilon}{2} \|f\|^2,~f \in \H.
\end{align}
\end{Df}

\begin{Df} \label{df_solution_Cauchy problem}
Let $(u(t))_{t\geq 0}$ denotes the classical solution to the abstract Cauchy problem $\dot{u}  = Lu$ in $\H$ with initial condition $g \in D(L)$, i.e., the map $t \mapsto u(t) \in \H$ is continuously differentiable on $[0,\infty)$, we have $u(t) \in D(L)$ for $t \geq 0$ and $u$ satisfies
\begin{align*} 
\frac{d}{dt}u(t) =Lu(t) \,\mbox{ for  all } t \geq 0,~u(0)=g.
\end{align*}
It is well-known that $u(t)$ is uniquely given by $u(t)=T(t)g$ for every $t \geq 0$. Since $\D$ is a core for $(L,D(L))$, the defining property in (D6) carries over to all elements from $D(L)$. Thus we conclude
\begin{align} \label{eq_invariant_measure}
\int_E T(t)g~\mathrm{d}\mu = \left(u(t),1\right)_\H=\left(u(0),1\right)_\H=\int_E g~\mathrm{d}\mu 
\end{align}
for all $t \geq 0$. This justifies the name \textit{invariant measure} in (D6). Now let $g$ still be an element from $D(L)$ and define $f(t)$ as
\begin{align*}
f(t)=u(t) - \left( g, 1\right)_{\H} ~\mbox{ for all } t \geq 0.
\end{align*}
Then by (D7), it follows that $(f(t))_{t \geq 0}$ is the unique classical solution to the abstract Cauchy problem $\dot{u}  = Lu$ in $\H$ with initial condition $g -  \left( g, 1\right)_{\H} \in D(L)$. By (D1) and \eqref{eq_invariant_measure} we get
\begin{align} \label{scalar_product_gt}
\left(f(t),1\right)_\H =0 ~\mbox{ for all } t \geq 0.
\end{align}
In case $g \in H$, the mapping $[0,\infty) \ni t \mapsto T(t)g \in \H$ is the mild solution of the abstract Cauchy problem $\dot{u}  = Lu$ in $\H$ with initial condition $g \in \H$. Note that the property from \eqref{eq_invariant_measure} is still satisfied since $D(L)$ is dense in $\H$.
\end{Df}

\begin{Rm}
(D7) also implies $T(t)1=1$ for all $t \geq 0$ since $\frac{d}{dt} T(t)1=T(t)L1=0$ for all $t \geq 0$ and $T(0)1=1$. This justifies the name \textit{conservativity} in (D7).
\end{Rm}

\begin{Df} \label{Df_entropy_dissipation_functional}
Let $0 \leq \varepsilon < 1$. Define the \textit{entropy dissipation functional} $\mathrm{D}_\varepsilon$ as in \cite{DMS10} via
\begin{align*}
\mathrm{D}_{\varepsilon}[t]:=-\frac{d}{dt} \mathrm{H}_\varepsilon[f(t)] ~\mbox{ for all }t \geq 0.
\end{align*}
Here $(f(t))_{t \geq 0}$ is the classical solution to the abstract Cauchy problem $\dot{u}  = Lu$ in $\H$ with initial value $g-\left(g,1\right)_{\H} \in D(L)$ as in Definition \ref{df_solution_Cauchy problem}. Thus we get
\begin{align*}
\mathrm{D}_\varepsilon[t] =  \mathrm{I}_1[f(t)] - \varepsilon ~\mathrm{I}_2[f(t)] - \varepsilon ~\mathrm{I}_3[f(t)]
\end{align*}
for every $t \geq 0$. Here the $\mathrm{I}_n:D(L) \rightarrow \mathbb{R}$ are defined by
\begin{align*}
&\mathrm{I}_1[f]=  -\left( Lf,f\right)_\H,~\mathrm{I}_2[f]= \left( B L f, f\right)_\H,~\mathrm{I}_3[f]= \left(B f , Lf\right)_\H,~f \in D(L).
\end{align*}
\end{Df}

After these definitions, we obtain the following lemma.

\begin{Lm} \label{Lm_1}
Assume that (H1) holds. Then we have
\begin{align*}
\mathrm{I}_3[f] \leq \|(I-P)f\|\|f\|,~f \in D(L).
\end{align*}
\end{Lm}

\begin{proof}
Let first $f \in \D$. Then by (D4), $Lf=Sf - Af$. Moreover, see Lemma \ref{Boundedness_of_B}, we have $P B=B$. Thus by (D5) it holds $\mathcal{R}(B) \subset D(S)$ and $SB=0$ on $\H$. So $\left( Bf, Sf\right)_\H = 0$ since $(S,D(S))$ is symmetric. As seen in the proof of Lemma \ref{Boundedness_of_B}, it holds $B(\D) \subset D(A)$. Hence
\begin{align*}
\mathrm{I}_3[f] =  -\left(B f , Af\right)_\H =  \left(A B f , f\right)_\H \leq \|(I-P)f\| \|f\|
\end{align*}
by the antisymmetry of $(A,D(A))$ and Lemma \ref{Boundedness_of_B}. Thus the claim follows since $\D$ is a core for $(L,D(L))$.
\end{proof}

The following microscopic coercivity condition as defined in \cite[Sec.\,1.3]{DMS10} needs again only to be introduced on the fixed operator core of $L$ in our setting.

\begin{H2} \label{H2} (Microscopic coercivity). There exists $\Lambda_m > 0$ such that
\begin{align*} 
-\left( Sf,f \right)_\H \geq \Lambda_m \, \|(I-P_S)f\|^2,~f \in \D. 
\end{align*}
\end{H2}

\begin{Lm} \label{Lm_2} Assume that (H2) holds. Then we have
\begin{align*}
\mathrm{I}_1[f] \geq \Lambda_m \, \|(I-P)f\|^2
\end{align*}
for all $f \in D(L)$ with $\left(f,1\right)_\H=0$.
\end{Lm}

\begin{proof}
For $f \in D(L)$ there exists $f_n \in \D$, $n \in \mathbb{N}$, with $f_n \rightarrow f$ and $Lf_n \rightarrow Lf$ as $n \rightarrow \infty$ with convergence in $\H$ due to (D3). If additionally $\left(f,1\right)_\H=0$, note that $\lim_{n \rightarrow \infty} P_S \, f_n = P f$ in $\H$. Hence
\begin{align*}
I_1[f] = - \lim_{n \rightarrow \infty} \left(Sf_n,f_n\right)_\H \geq \Lambda_m \lim_{n \rightarrow \infty} \|(I-P_S) f_n\|^2 = \Lambda_m \, \|(I-P)f\|^2
\end{align*}
where we have used that $\left(Af_n,f_n\right)_\H=0$ since $(A,D(A))$ is antisymmetric.
\end{proof}

Before introducing (H3), we shall discuss one more lemma first.

\begin{Lm} \label{Lm_before_H3}
The operator $I - P A^2P: \D \rightarrow \H$ has dense range if and only if $\D$ is a core for the operator $(AP)^*AP$ with domain $D((AP)^*AP)$. 
\end{Lm}

\begin{proof}
For notational convenience, we set $T=AP$ in this proof. Assume that the operator $I - P A^2P: \D \rightarrow \H$ has dense range. Clearly, the operator $(I - P A^2P,\D)$ is symmetric and semibounded with lower bound $\alpha \geq 1$. Hence $(I - P A^2P,\D)$ is closable with closure $(K,D(K))$ which is also semibounded with lower bound $\alpha \geq 1$. In particular, $(K,D(K))$ is injective, thus $K^{-1}:\mathcal{R}(K) \rightarrow \H$ is also closed. But $K^{-1}$ is also bounded with operator norm less or equal than $1$. Hence $D(K^{-1})=\mathcal{R}(K)$ must be a closed subset of $\H$. Consequently, $\mathcal{R}(K)=\overline{\mathcal{R}(K)}=\H$ by assumption. Moreover, by von Neumann's theorem, we know that the operator 
\begin{align*}
I+T^*T:D(T^*T) \rightarrow \H
\end{align*}
is bijective and extends  $(I - P A^2P,\D)$. Let $f \in D(T^*T)$ and define $g:= (I+T^*T)f$ which is an element of $\mathcal{R}(K)$. Thus there exists $f_n \in \D$, $n \in \mathbb{N}$, such that
\begin{align*}
\lim_{n \rightarrow \infty} (I+T^*T)f_n = \lim_{n \rightarrow \infty} (I- P A^2P)f_n = g = (I+T^*T)f
\end{align*}
Applying the continuous operator $(I+T^*T)^{-1}$ on both sides yields $\lim_{n \rightarrow \infty} f_n = f$. So $\D$ is a core for $(I+T^*T,D(T^*T))$. Thus $\D$ is a core for $(T^*T,D(T^*T))$. The reverse direction is now obvious, just repeat some of the previous arguments.
\end{proof}

For the following macroscopic coercivity assumption as originally introduced in \cite{DMS10}, it does not suffice to verify it only on $\D$, see the proof of Lemma \ref{Lm_3} below. Nevertheless, we give a sufficient criterion for (H3) to hold in the upcoming corollary.

\begin{H3} \label{H3} (Macroscopic coercivity). There exists $\Lambda_M > 0$ such that
\begin{align} \label{eq_inequality_D3} 
\|AP f\|^2 \geq \Lambda_M \|P f \|^2,~f \in D((AP)^*(AP)).
\end{align}
\end{H3}

Now here is a sufficient condition implying (H3). This criterion is later on used in our applications.

\begin{Cor} \label{Cor_H3_equivalent} Assume that the inequality in \eqref{eq_inequality_D3} holds for all $f \in \D$ and assume that the operator $(P A^2P,\D)$ is essentially m-dissipative (or essentially selfadjoint respectively), this is, $(I - P A^2P)(\D)$ is dense in $\H$. Then (H3) is satisfied.
\end{Cor}

\begin{proof}
First of all, by the conditions in (D) note that the operator $(P A^2P,\D)$ is symmetric and nonpositive definite, in particular, it is dissipative. Moreover, recall that the property of some symmetric, densely defined and nonpositive definite operator to be essentially m-dissipative (i.e., its closure is m-dissipative) is equivalent to essential selfadjointness. And the Lumer-Phillips theorem now implies that $(P A^2P,\D)$ is essentially m-dissipative if and only if $(\alpha I - P A^2P)(\D)$ is dense in $\H$ for one (hence all) $\alpha > 0$, see e.g.~\cite{Gol85}. This clarifies our notations.\\
Now rewrite the left hand side of \eqref{eq_inequality_D3} in the form $\left( (AP)^*AP f, f\right)_\H$ for $f \in \D$ or $f \in D((AP)^*(AP))$. The claim follows by Lemma \ref{Lm_before_H3}.
\end{proof}

Now the final assumption. Compare it again with (H4) in \cite{DMS10}.

\begin{H4} \label{H4} (Boundedness of auxiliary operators). 
The operators $(BS,\D)$ and $(BA(I-P),\D)$ are bounded and there exists constants $N_1 < \infty$ and $N_2 < \infty$ such that
\begin{align*} 
\|BSf\| \leq N_1 \,\|(I-P_1)f\|,~\|BA(I-P)f\| \leq N_2 \,\|(I-P_2)f\|
\end{align*}
holds for all $f\in \D$. Here $P_n$ is either equal to $P$ or $P_S$ for each $n=1,2$.
\end{H4}

Analogously as for Assumption (H3), we shall give sufficient criteria implying (H4) which can be used in applications. The sufficient criteria are contained in Lemma \ref{Pp_suff_H4_part1} and Proposition \ref{Pp_suff_H4_part2} below and are motivated by the considerations and calculations done in \cite[Sec.~1.5, Sec.~2, Sec.~3]{DMS10} and in \cite{DKMS11}.

\begin{Lm} \label{Pp_suff_H4_part1}
Let (H1) holds. Assume that $S(\D) \subset D(A)$ and assume that there exists some $M_1 \in \mathbb{R}$ such that
\begin{align*}
PAS = M_1 \,PA ~~\mbox{ on } \D.
\end{align*}
Then $(BS,\D)$ is bounded and the first inequality in (H4) holds with $N_1= \frac{1}{2} |M_1|$.
\end{Lm}

\begin{proof}
By assumption, note that $(AP)^*S=M_1 \,(AP)^*$ on $\D$, see Lemma \ref{Lm_APi}.  Now the representation of $B$ on $D((AP)^*)$ implies that $BS = M_1 \, B$ on $\D$. So the claim follows by applying Estimate \eqref{norm_of_B} from Lemma \ref{Boundedness_of_B}.
\end{proof}

For the second sufficient criteria it is interesting to note that again essential selfadjointness of $(P A^2P,\D)$ is needed, analogously as in Corollary \ref{Cor_H3_equivalent} before.

\begin{Pp} \label{Pp_suff_H4_part2}
For each $g=(I-PA^2P)f$, $f \in \D$, we have
\begin{align} \label{g_in_DBA*}
g \in D((BA)^*) ~\mbox{ with }~(BA)^*g=-A^2Pf.
\end{align}
Assume that $(PA^2P,\D)$ is essentially m-dissipative (or essentially selfadjoint respectively) and assume that there exists $M_2 < \infty$ such that
\begin{align} \label{eq_boundedness_BA*}
\|(BA)^*g\| \leq M_2 \,\|g\| ~~\mbox{ for all }~ g=(I-PA^2P)f,~f \in \D.
\end{align}
Then the second inequality in (H4) is satisfied with $N_2=M_2$.
\end{Pp}

\begin{proof}
By Lemma \ref{Lm_basics_closed_operators}, we have
\begin{align*}
((BA)^*,D((BA)^*))=(A^*B^*,D(A^*B^*)).
\end{align*}
So let $g \in \H$ be of the form $g=(I-PA^2P)f$ for some $f \in \D$ and choose $h \in D((AP)^*)$. By the representation of $B$ on $D((AP)^*)$, the fact that the adjoint of $(I+(AP)^*AP)^{-1}$ is equal to $(I+(AP)^*AP)^{-1}$ and since $f \in D(AP)$ we get
\begin{align*}
\left(h,B^*g\right)_\H=\left(Bh,g\right)_\H=\left((AP)^*h,f\right)_\H=\left(h,APf\right)_\H.
\end{align*}
Thus $B^*g=APf \in D(A) \subset D(A^*)$. Hence $A^*B^*g=-A^2Pf$ by Lemma \ref{Lm_APi}. So \eqref{g_in_DBA*} is shown.
Now by \eqref{eq_boundedness_BA*}, for the closed operator $((BA)^*,D((BA)^*))$ we must have that the closure of $(I-PA^2P)(\D)$ in $\H$ is contained in $D((BA)^*)$. But $(I-PA^2P)(\D)$ is dense in $\H$ by assumption. Hence $(BA)^*$ is already a bounded operator on $\H$ by the closed graph theorem. Finally, by Lemma \ref{Lm_basics_closed_operators}, the operator $(BA,D(A))$ is closable, its closure $\overline{BA}$ is a bounded operator on $\H$ with $\|\overline{BA}\|=\|(BA)^*\|$. In particular, 
\begin{align*} 
&\|BA(I-P)f\| \leq M_2 \, \|(I-P)f\|~~ \mbox{ for all } f \in \D.
\end{align*}
Hence the proof is finished.
\end{proof}

We remark that Inequality \eqref{eq_boundedness_BA*} can later on be verified in our application with the help of some a priori estimate for a suitable elliptic equation obtained by Dolbeault, Mouhot and Schmeiser, see Section \ref{Hypocoercivity_section_Spherical_velocity_Langevin}. But now, one more lemma.

\begin{Lm} \label{Lm_3} Assume that (H3) and (H4) holds. Then
\begin{align*}
\mathrm{I}_2[f] \leq - \frac{\Lambda_M}{1+\Lambda_M} \|P f\|^2 + (N_1+N_2) \,\| (I-P)f\| \|P f\| 
\end{align*}
for all $f \in D(L)$ with $\left(f,1\right)_\H=0$.
\end{Lm}

\begin{proof}
Let first $f \in \D$. Then
\begin{align} \label{eq_I2_on_D}
\mathrm{I}_2[f] &= \left( B S f, f\right)_\H - \left( B A (I-P) f, f\right)_\H - \left( B A P f, f\right)_\H \nonumber \\
&=\left((B S -  B A (I-P))f, P f\right)_\H - \left( B A P f, P f\right)_\H. 
\end{align}
By (D5) we have $A P f \in D(A) \subset D((AP)^*)$. Hence by the representation of $B$ on $D((AP)^*))$, we obtain
\begin{align} \label{Eq_BAPi}
BA P f = (I+(AP)^*A P)^{-1}(AP)^*AP f &= P f - (I+(AP)^*A P)^{-1} \,P f
\end{align}
Moreover, it is easy to see that $I+(AP)^*A P$ maps $D((AP)^*A P) \cap \mathcal{R}(P)$ bijectively to $\mathcal{R}(P)$.
Now let $g \in D((AP)^*A P)$ such that $P g=g$. Then (H3) implies
\begin{align*}
\left( (I+(AP)^*A P) g, g \right)_\H = \| g \|^2 + \|A P g\|^2 \geq (1+\Lambda_M) \, \|g\|^2
\end{align*}
Consequently,
\begin{align*}
\| (I+(AP)^*A P)^{-1} h\| \leq \frac{1}{1 + \Lambda_M} \|h\| ~\mbox{ for all } h \in \mathcal{R}(P).
\end{align*}
Hence Equation \eqref{Eq_BAPi} gives
\begin{align*}
\left(BA P f, P f \right)_\H \geq \frac{\Lambda_M}{1+\Lambda_M} \|P f\|^2 ~\mbox{ for all } f \in \D.
\end{align*}
By using additionally the representation \eqref{eq_I2_on_D} and (H4), we obtain
\begin{align*}
\mathrm{I}_2[f] \leq - \frac{\Lambda_M}{1+\Lambda_M} \|P f\|^2 + N_1 \| (I-P_1)f\| \|P f\| + N_2 \| (I-P_2)f\| \|P f\|
\end{align*}
for all $f \in \D$. And since $\D$ is a core for $D(L)$ the last inequality carries over to all elements $f \in D(L)$, in particular, is satisfied for all $f \in D(L)$ with $\left(f,1\right)_\H=0$. Thus the claim follows since $P_n f=P f$ for all such $f$ and each $n=1,2$.
\end{proof}

\begin{Rm}
The previous proof shows that instead of (H4) we may (for instance) assume that
\begin{align*}
\left(B (S -  A (I-P))f, P f\right)_\H \leq N\, \| (I-P) f \| \| P f \|
\end{align*}
for some constant $N \in [0,\infty)$ and all $f \in \D$. Nevertheless, (H4) as formulated above seems to be more suitable for the applications.
\end{Rm}

Now, we arrive at the final hypocoercivity theorem. It can now be proven, almost word by word, completely analogous as the original statement in the framework of \cite{DMS10}, see \cite[Theo.~2]{DMS10}. Just for completeness, we include this proof here. Notice the additional assumption (D) in our setting.

\begin{Thm} \label{Thm_Hypocoercivity}
Assume that (D) and (H1)-(H4) holds. Then there exists strictly positive constants $\kappa_1 < \infty$ and $\kappa_2< \infty$ which are explicitly computable in terms of $\Lambda_m,~\Lambda_M,~N_1$ and $N_2$ such that for all $g \in \H$ we have
\begin{align*}
\left\|T(t)g - \left(g,1\right)_\H\right\| \leq \kappa_1 e^{-\kappa_2 \,t}  \left\|g - \left(g,1\right)_\H\right\| ~\mbox{ for all } t \geq 0.
\end{align*}
\end{Thm}

\begin{proof}
Let first $g \in D(L)$ and let $\mathrm{D}_\varepsilon[\cdot]$ and $(f(t))_{t \geq 0}$ be as in Definition \ref{Df_entropy_dissipation_functional}. To show: There exists a positive constant $\kappa > 0$ and some suitable $0 < \varepsilon <1$ (both independent of $g$) such that 
\begin{align} \label{to_show_coercivity_of_D}
\mathrm{D}_\varepsilon [t] \geq \kappa \,\|f(t)\|^2
\end{align}
holds for all $t \geq 0$. Indeed, assume this to be true. By using \eqref{equivalence_H_norm} we obtain
\begin{align*}
\frac{d}{dt} \mathrm{H}_\varepsilon[f(t)] \leq -\frac{2 \kappa}{1 + \varepsilon} \,\mathrm{H}_\varepsilon[f(t)] ~~\mbox{ for all } t \geq 0.
\end{align*}
Gronwall's lemma and \eqref{equivalence_H_norm} then implies the claim for $g \in D(L)$ with $\kappa_1=\sqrt{\frac{1+\varepsilon}{1-\varepsilon}}$ and $\kappa_2 =\frac{\kappa}{1+\varepsilon}$. So let us verify the coercivity property \eqref{to_show_coercivity_of_D}. For notational convenience, we set $f_t:=f(t)$ for $t \geq 0$. Lemma \ref{Lm_1}, Lemma \ref{Lm_2} and Lemma \ref{Lm_3} in combination with Identity \eqref{scalar_product_gt} imply
\begin{align*}
\mathrm{D}_\varepsilon[t] &\geq \Lambda_m \|(I-P)f_t\|^2 + \varepsilon \frac{\Lambda_M}{1+\Lambda_M} \|P f_t\|^2 - \varepsilon (1+N_3)\, \| (I-P)f_t\| \| f_t\| \\
& \geq \left( \Lambda_m - \varepsilon (1+N_3) \frac{1+\delta^2}{2 \delta} \right) \|(I-P)f_t\|^2 \\ &~~~~~+ \varepsilon \left(  \frac{\Lambda_M}{1+\Lambda_M} - (1+N_3) \frac{\delta}{2} \right) \|P f_t\|^2
\end{align*}
where $N_3=N_1+N_2$ and $\delta > 0$ is arbitrary. By choosing first $\delta$ and then $\varepsilon$ small enough, \eqref{to_show_coercivity_of_D} follows. So the statement of the theorem is shown in case $g \in D(L)$. Now note that the rate of convergence in terms of $\kappa_1$ and $\kappa_2$ is independent of $g \in D(L)$. Hence use the density of $D(L)$ in $\H$ to finish the proof.
\end{proof}
\section{Hypocoercivity and exponential convergence to equilibrium of a spherical velocity Langevin process} \label{Hypocoercivity_section_Spherical_velocity_Langevin}

In the following, let $d \in \mathbb{R}^d$ be fixed with $d \geq 2$ and let $\sigma$ be a finite constant satisfying $\sigma \geq 0$. In this section we study the ergodic behavior of the spherical velocity Langevin process \eqref{Fiber_Model_Intro} with state space $\mathbb{M}=\mathbb{R}^d \times \mathbb{S}$ whose associated  Kolmogorov backward operator writes in the form
\begin{align} \label{Kolmogorov_operator_spherical_velocity_Langevin}
L= \omega \cdot \nabla_x - \text{grad}_{\mathbb{S}} \Phi \cdot \nabla_\omega + \frac{1}{2} \sigma^2 \, \Delta_{\mathbb{S}}\,~\mbox{ with }\,~\Phi(x,\omega)= \frac{1}{d-1} \,\nabla_x V(x) \cdot \omega.
\end{align}
For notations we refer to the introduction. As mentioned before, the spherical velocity Langevin process is in some sense the analogue of the classical Langevin process moving with velocity of euclidean norm equal to $1$. Moreover, exactly this equation arises in industrial applications as the so-called fiber lay-down process in modeling (virtual) nonwoven webs, see Section \ref{section_Introduction} for further motivation. For illustration and visualization consider Figure \ref{figure_nonwovenweb} below, where two different fiber webs are generated with the program \glqq SURRO\grqq~that has been developed at the Fraunhofer ITWM in Kaiserslautern (department \glqq Transport Processes\grqq). Therein, the two-dimensional version of the spherical velocity Langevin process (including a moving conveyor belt) is implemented. For more details on all these industrial applications, we refer to \cite{GKMW07}, \cite{KMW09}, \cite{HM05}, \cite{MW11} and \cite{MW06}.

\begin{figure}[hbp] 
\subfigure{\includegraphics[scale=0.53]{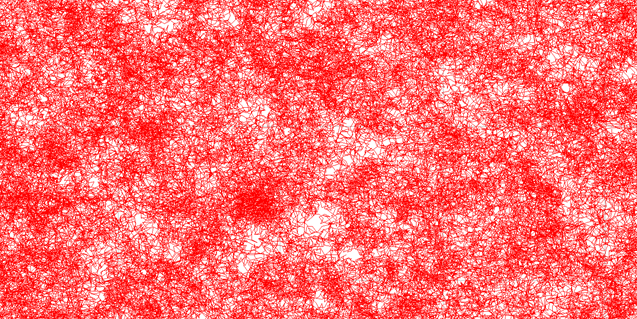}} \hspace{0,36cm} 
\subfigure{\includegraphics[scale=0.53]{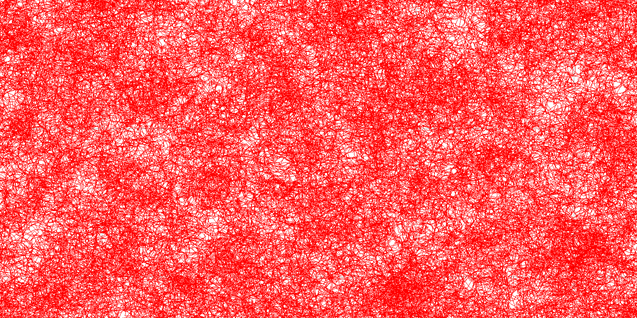}} %
\caption{Industrial usage of the spherical velocity Langevin equation in simulating nonwoven webs} \label{figure_nonwovenweb}
\end{figure}

Of particular interest is the long-time behavior of the spherical velocity Langevin process, especially its exponential decay to equilibrium (see again Section \ref{section_Introduction}). Thus we make use of the hypocoercivity strategy developed in the previous section. Let us remark again that in \cite{GKMS12}, we have already applied the original hypocoercivity strategy from \cite{DMS10} in an algebraic way and have not yet discussed any domain issues. Hence a rigorous elaboration of the calculations has been left open, see \cite[Rem.~6.1]{GKMS12} and \cite[Rem.~6.5]{GKMS12}. In the following, we apply the methods developed in Section \ref{Section_Extension_Hypocoercivity} in a mathematical precise way and finally end up with the desired hypocoercivity theorem. 

Therefore, we recall first some useful formulas and introduce basic notations and our necessary geometric language. Afterwards, we start constructing the setup in Section \ref{Sec_the_data_(D)}, this is, verify and prove the conditions (D1)-(D7). Here the construction of the strongly continuous semigroup associated to the Kolmogorov operator $L$ from \eqref{Kolmogorov_operator_spherical_velocity_Langevin} uses powerful functional analytic tools such as the hypoellipticity of $L$ as well as Kato pertubation and is done in Section \ref{Section_Construction_Semigroup} separately. We mention that we do not restrict on the case $V \in C^\infty(\mathbb{R}^d)$ and can even solve the abstract Cauchy problem for $L$ under that assumption that $V$ is just locally Lipschitz continuous and bounded from below. Finally, see Section \ref{Section_(H1)_(H4)}, we prove the desired conditions (H1)-(H4) completing the hypocoercivity program of Section \ref{Section_Extension_Hypocoercivity}.

\subsection{Some more notations and formulas} \label{Section_Formulas_Notations}
First some notations. Let $\mathbb{X}$ be a smooth manifold. $\mathbb{X}$ is always equipped with its natural Borel-sigma algebra $\mathcal{B}(\mathbb{X})$ and induced Riemannian volume measure in case measurability issues are considered. $C^n(\mathbb{X})$, $n \in \mathbb{N} \cup \{\infty\}$, denotes the set of all $n$-times continuously differentiable functions $f:\mathbb{X} \rightarrow \mathbb{R}$. The index $c$ denotes compact support. The support of such an $f$ is denoted by $\text{supp}(f)$. As usual $H^{m,p}(\mathbb{R}^d)$, $m \in \mathbb{N}$, $p \in [1,\infty]$, denotes the Soboloev space of all $m$-times weakly differentiable functions with $L^p(\mathbb{R}^d,\mathrm{d}x)$-integrable weak derivatives (including the function itself). The notation $f \in Z_{\text{loc}}$, $Z$ some Sobolev or Lebesgue space, is understood in the sense that $f$ is assumed to be measurable and is locally (on relative compact open subsets of the underlying state space) an element from $Z$. For $x \in \mathbb{R}^d$, the open ball with radius $r>0$ (with respect to $| \cdot |_{\text{euc}}$) around $x$ is denoted by $B_r(x)$. Additionally to the notation from the Introduction, we denote partial derivatives in euclidean space by $\partial_{x_i}$ for $i=1,\ldots,d$ and $\Delta_x$ (or $\Delta$) is the Laplace operator in $\mathbb{R}^d$. 

We need some more notation and geometric language, similar as in \cite{GKMS12}. So consider the closed regular submanifold $\mathbb{M}=\mathbb{R}^d \times \mathbb{S}$ of $\mathbb{R}^{2d}$ introduced previously. Let $X:\mathbb{M} \rightarrow \mathbb{R}^{2d}$ be measurable and let $X$ be tangential to $\mathbb{M}$ almost everywhere, i.e., it holds 
\begin{align*}
X(x,\omega) \cdot \begin{pmatrix} 0 \\ \omega \end{pmatrix} =0~~\mbox{ for a.e.}~(x,\omega) \in \mathbb{M}.
\end{align*}
Hence $X(x,\omega)$ is an element from the (algebraic) tangent space at the point $(x,\omega)$ for a.e.~$(x,\omega) \in \mathbb{M}$, that is, $X(x,\omega)$ induces a derivation of the form
\begin{align} \label{Df_vector_field_almost_everywhere}
X(x,\omega)(f):=Xf(x,\omega):=X(x,\omega) \cdot \nabla_{(x,\omega)} \widetilde{f}
\end{align}
where $f \in C^\infty(\mathbb{M})$ and $\widetilde{f} \in C^\infty(\mathbb{R}^{2d})$ is arbitrary chosen such that $\widetilde{f}$ extends $f$ in some open neighbourhood (in $\mathbb{M}$) of $(x,\omega)$. This is indeed well-defined. Here $\nabla_{(x,\omega)}$ is the usual gradient in $\mathbb{R}^{2d}$. Moreover, recall that for any $f \in C^\infty(\mathbb{M})$ one can construct an $\widetilde{f} \in C^\infty(\mathbb{R}^{2d})$ which extends $f$ on the whole of $\mathbb{M}$. Altogether, for every $f \in C^\infty(\mathbb{M})$ we conclude that the function $Xf:\mathbb{M} \rightarrow \mathbb{R}$, defined by $(x,\omega) \mapsto Xf(x,\omega)$, is measurable and vanishes on the complement of $\text{supp}(f)$. $X$ is also called a vector field on $\mathbb{M}$ and by abuse of notation we canonically identify $X$ with $X \cdot \nabla_{(x,\,\omega)}$, in notation $X \equiv X \cdot \nabla_{(x,\,\omega)}$. Then, of course,
\begin{align*}
X \cdot \nabla_{(x,\,\omega)} f,~f \in C^\infty(\mathbb{M}),
\end{align*} 
is understood as introduced in \eqref{Df_vector_field_almost_everywhere}. The previous considerations are valid for any closed regular submanifold embedded in euclidean space. Next, the spherical gradient of some $F \in C^\infty(\mathbb{S})$ is denoted by $\text{grad}_{\mathbb{S}}\,F:\mathbb{S} \rightarrow \mathbb{R}^d$ and  is given by
\begin{align*}
\text{grad}_{\mathbb{S}}\,F(\omega)= (I-\omega \otimes \omega) \,\nabla_\omega \widetilde{F} (\omega),~\omega \in \mathbb{S},
\end{align*}
where $\widetilde{F} \in C^\infty(\mathbb{R}^d)$ extends $F$ locally in some open neighbourhood in $\mathbb{S}$ around $\omega$. Here recall that $x \otimes y:=x y^T$, $x,y \in \mathbb{R}^{d}$ and $y^T$ denotes the transpose of $y$. This definition is again independent of the local smooth extension $\widetilde{F}$ for $F$. In short notation we again write  
\begin{align*}
\text{grad}_{\mathbb{S}}\,F= (I-\omega \otimes \omega) \,\nabla_{\omega} F,~F\in C^\infty(\mathbb{S}).
\end{align*}
And since $\text{grad}_{\mathbb{S}}\,F$ is a vector field on $\mathbb{S}$, embedded in $\mathbb{R}^d$, we change notation and canonically identify it again pointwisely with its associated derivation, i.e.,
\begin{align*}
\text{grad}_{\mathbb{S}}\,F \equiv \text{grad}_{\mathbb{S}}\,F \cdot \nabla_\omega = (I-\omega \otimes \omega) \,\nabla_{\omega} F \cdot \nabla_\omega,
\end{align*}
analogously as in the $X$-case above. Recall that, for $f,g \in C^\infty(\mathbb{S})$ one has
\begin{align*}
\left( \text{grad}_{\mathbb{S}}\,f , \text{grad}_{\mathbb{S}}\,g \right)_{T\mathbb{S}} = \left( \text{grad}_{\mathbb{S}}\,f , \text{grad}_{\mathbb{S}}\,g  \right)_{\text{euc}}
\end{align*}
where $\left( \cdot, \cdot \right)_{T\mathbb{S}}$ denotes the induced Riemannian scalar product on the tangent bundle $T\mathbb{S}$ of $\mathbb{S}$.

Moreover, we need the following representation of the Laplace Beltrami operator on $\mathbb{S}$. So let $e_n$ be the $n$-th unit vector in $\mathbb{R}^d$, $n=1,\ldots,d$, and define the spherical vector fields $\mathcal{S}_n:\mathbb{S} \rightarrow \mathbb{R}^d$ as
\begin{align} \label{Def_Sn}
\mathcal{S}_n \equiv (I-\omega \otimes \omega) e_n \cdot \nabla_\omega
\end{align}
Then the spherical Laplace Beltrami is given as $\Delta_\mathbb{S} = \sum_{n=1}^d \mathcal{S}_n^2$. An even more general representation formula can be found in \cite[Theo.~3.1.4]{Hsu02}.

Finally, recall that $\nu$ denotes the normalized Riemannian volume measure of $\mathbb{S}$, i.e., $\nu={\text{vol}(\mathbb{S})}^{-1} \mathcal{S}$ where $\text{vol}(\mathbb{S})=\mathcal{S}(\mathbb{S})$ is the surface area of $\mathbb{S}$ and $\mathcal{S}$ the Riemannian volume measure of $\mathbb{S}$. For later use, we now recall some formulas which can be proven using the Gaussian integral formula, see e.g.~\cite[Sec.~7]{GKMS12}.

\begin{Lm} \label{Gauss_formula_sphere}
Let $B$ be a matrix with entries $b_{ij} \in \mathbb{R}$ for all $i,j=1, \ldots, d$. Moreover, let $z,z_1,z_2 \in \mathbb{R}^{d}$. It holds
\begin{align*}
&\int_{\mathbb{S}} \left( z, \omega \right)_{\text{euc}} \mathrm{d}\nu(\omega)=0,~\int_{\mathbb{S}} \left( B \omega, \omega \right)_{\text{euc}} \mathrm{d}\nu(\omega) =  \frac{1}{d} \sum_{i=1}^{d} b_{ii},\\
&\int_{\mathbb{S}} \left(  z_1, \omega \right)_{\text{euc}} \left(  z_2, \omega \right)_{\text{euc}} \mathrm{d}\nu(\omega) =  \frac{1}{d} \left( z_1, z_2 \right)_{\text{euc}}.
\end{align*}
\end{Lm}

Furthermore, we need the following identity, see again e.g.~\cite[Lem.~7.1]{GKMS12}.

\begin{Lm} \label{Lm_formula_sphere_strong_feller}
Let $I_{\mathbb{S}}$ denotes the function $\omega \mapsto \omega$, $\omega \in \mathbb{S}$, where $\Delta_{\mathbb{S}}I_{\mathbb{S}}$ is understood componentwise. Then it holds
\begin{align*}
\Delta_{\mathbb{S}} I_{\mathbb{S}}= - (d-1) \,I_{\mathbb{S}}.
\end{align*}
\end{Lm}

The following lemma is also contained in \cite{GKMS12}. For completeness, we shall repeat the argument here.

\begin{Lm} \label{Lm_phi_sphere}
Let  $\phi:\mathbb{S} \rightarrow \mathbb{R}$ be defined by $\phi(\omega)= \left( z , \omega \right)_{\text{euc}}$, $\omega \in \mathbb{S}$, where $z \in \mathbb{R}^{d}$ is fixed. Define the vector field $X$ by $X= \text{grad}_{\mathbb{S}} \,\phi$. Then it holds
\begin{align*}
\int_{\mathbb{S}} X(f) \, \mathrm{d}\nu = \left(d-1\right) \int_{\mathbb{S}}  f \, \phi  \,\mathrm{d} \nu ,~f \in C^\infty(\mathbb{S}).
\end{align*}
\end{Lm}

\begin{proof}
Let $f \in C^\infty(\mathbb{S})$. Then it is easy to see that $X(f)=\left(\text{grad}_{\mathbb{S}} \phi, \text{grad}_{\mathbb{S}} f\right)_{T\mathbb{S}}$. Hence we have
\begin{align*} 
\int_{\mathbb{S}} X(f) \, \mathrm{d}\nu = \int_{\mathbb{S}} \left(\text{grad}_{\mathbb{S}} \phi, \text{grad}_{\mathbb{S}} f\right)_{T\mathbb{S}} \, \mathrm{d}\nu = - \int_{\mathbb{S}} \left( \Delta_{\mathbb{S}} \phi \right) f \, \mathrm{d}\nu.
\end{align*}
by Green's formula. The claim follows since $\Delta_{\mathbb{S}}\phi  = -\left(d-1\right) \phi$ by Lemma \ref{Lm_formula_sphere_strong_feller}.
\end{proof}

\subsection{The data (D)} \label{Sec_the_data_(D)}

Now we start introducing and verifying the conditions (D) from Section \ref{Section_Extension_Hypocoercivity}. First recall that if $f$ is locally Lipschitz continuous on $\mathbb{R}^d$, then it is differentiable a.e.~on $\mathbb{R}^d$ and $\nabla f \in L^\infty_{\text{loc}}(\mathbb{R}^d)$, see e.g.~\cite[Satz~8.5]{Alt06}. First we introduce the Hilbert space and our desired Kolmogorov backward operator under weak continuity assumption on the potential $V$.

\begin{Df} \label{Df_operator_weak_ass_potential} The potential $V:\mathbb{R}^d \rightarrow \mathbb{R}$ is assumed to be locally Lipschitz continuous. The measure space $(\mathbb{M},\mathcal{B}(\mathbb{M}),\mu)$ and Hilbert space $\H$ is defined as 
\begin{align*}
\mu=e^{-V}\,\mathrm{d}x \otimes \nu,~\H=L^2(\mathbb{M},\mu).
\end{align*}
We introduce $\D$ as $\D=C^\infty_c(\mathbb{M})$ and let $\mathcal{A}:\mathbb{M} \rightarrow \mathbb{R}^{2d}$ be given as 
\begin{align} 
\mathcal{A}(x,\omega)= - \begin{pmatrix} \omega \\ - \frac{1}{d-1} \left(I- \omega \otimes \omega\right) \nabla V (x) \end{pmatrix} ~\mbox{ for a.e.~} (x,\omega) \in \mathbb{M}.
\end{align}
Define $(A,\D)$ as $Af=\mathcal{A}f$ for $f \in \D$. By the discussion in Section \ref{Section_Formulas_Notations}, note that really $Af \in \H$ and that $A$ is given on $\D$ as
\begin{align} \label{alternative_repr_A_smooth_case} 
A\equiv - \omega \cdot \nabla_x + \text{grad}_{\mathbb{S}}\Phi \cdot \nabla_\omega.
\end{align} 
Here $\Phi:\mathbb{M} \rightarrow \mathbb{R}$ is defined as 
\begin{align*}
\Phi(x,\omega)= \frac{1}{d-1} \left(\omega,\nabla V(x)\right)_{\text{euc}} \mbox{ for } \mbox{a.e.~}(x,\omega) \in \mathbb{M}.
\end{align*}
Finally, $(S,\D)$ and $(L,\D)$ are linear operators on $\H$ defined via
\begin{align*}
S=\frac{\sigma^2}{2}\,\Delta_\mathbb{S},~L=S-A ~~\mbox{ on } \D. 
\end{align*}
\end{Df}

Now the next statement summarizes some basic properties of $(L,\D)$.

\begin{Lm} \label{Lm_properties_of_L} Let $V:\mathbb{R}^d \rightarrow \mathbb{R}$ be locally Lipschitz continuous and let $(L,\D)$, $L=S-A$ on $\D=C_c^\infty(\mathbb{M})$, $\H=L^2(\mathbb{M},\mu)$ and the probability measure $\mu$ be as in Definition \ref{Df_operator_weak_ass_potential}. Then
\begin{itemize}
\item[(i)]
$(S,\D)$ is symmetric and nonpositive definite on $\H$.
\item[(ii)]
$(A,\D)$ is antisymmetric on $\H$.
\item[(iii)]
$\mu$ is invariant for $(L,\D)$ in the sense that $Lf \in L^1(\mathbb{M},\mu)$ for all $f \in \D$ and \eqref{Df_invariant_measure} is satisfied.
\end{itemize}
\end{Lm}

\begin{proof}
Let $f_1,f_2 \in C^\infty(\mathbb{M})$ with at least one of both having compact support. Then
\begin{align*}
\int_\mathbb{M} \left(\Delta_\mathbb{S} f_1\right) f_2 \, \mathrm{d}\mu = - \int_{\mathbb{R}^d}  e^{-V} \int_\mathbb{S} \left( \text{grad}_{\mathbb{S}}\, f_1, \text{grad}_{\mathbb{S}}\, f_2 \right)_{T\mathbb{S}} \, \mathrm{d}\nu \,   \mathrm{d}x = \int_\mathbb{M} f_1 \left(\Delta_\mathbb{S} f_2\right)\, \mathrm{d}\mu
\end{align*}
by using Green's formula and Fubini's theorem. In particular, (i) is shown. Now let $h \in \D$. Note that $\int_\mathbb{M} S h \, \mathrm{d}\mu=0$ by the last identity. Choose $g \in C^\infty_c(\mathbb{R}^d)$. Furthermore, note that $e^{-V}$ is also locally Lipschitz continuous, thus $e^{-V} \in H^{1,\infty}_{\text{loc}}(\mathbb{R}^d)$. This together with the weak Gaussian integral formula, see \cite[A.6.8]{Alt06}, implies 
\begin{align*}
\int_{\mathbb{R}^d} \partial_{x_i} \left( g \,e^{-V} \right) \, \mathrm{d}x =0,~i=1,\ldots,d.
\end{align*}
Consequently, we get
\begin{align} \label{eq1_properties_of_L}
\int_\mathbb{S} \int_{\mathbb{R}^d}\,\omega \cdot \nabla_x \left( h \, e^{-V} \right) \, \mathrm{d}x \,\mathrm{d}\nu =0.
\end{align}
By using Lemma \ref{Lm_phi_sphere} and Fubini's theorem note that
\begin{align} \label{eq2_properties_of_L}
\int_\mathbb{M} \text{grad}_\mathbb{S} \, \Phi (h) \, \mathrm{d}\mu = \int_\mathbb{M} h \, \left(\omega, \nabla V \right)_{\text{euc}} \, \mathrm{d}\mu. 
\end{align}
So from Identity \eqref{eq1_properties_of_L} and \eqref{eq2_properties_of_L} we can infer that $\int_\mathbb{M} Ah \, \mathrm{d}\mu =0$. Now let $h:=f_1f_2$ with $f_1$ and $f_2$ be as above. Then
\begin{align*}
\int_\mathbb{M} \left(Af_1\right) f_2 \, \mathrm{d}\mu + \int_\mathbb{M} f_1 \left(Af_2\right) \, \mathrm{d}\mu = \int_\mathbb{M} Ah \, \mathrm{d}\mu =0.
\end{align*}
Altogether, statement (ii) and (iii) follows.
\end{proof}

\begin{Df}
The previous lemma implies that $(S,\D)$ and $(A,\D)$ are densely defined and dissipative on $\H$, hence closable. The closures of $(S,\D)$ and $(A,\D)$ in $\H$ are denoted by $(S,D(S))$ and $(A,D(A))$. Then $(S,D(S))$ is symmetric and $(A,D(A))$ is antisymmetric. Moreover, $(L,\D)$ is also densely defined and dissipative on $\H$. The closure of $(L,\D)$ in $\H$ is denoted by $(L,D(L))$.
\end{Df}

Next, we introduce the desired orthogonal projections $P$ and $P_S$.

\begin{Df} As before, let $V:\mathbb{R}^d \rightarrow \mathbb{R}$ be locally Lipschitz continuous. First of all, $P_S:\H \rightarrow \H$ is defined as
\begin{align*}
P_S  f= \int_{\mathbb{S}} f\, \mathrm{d}\nu,~f \in \H.
\end{align*}
By using Fubini's theorem and the fact that $(\mathbb{S},\mathcal{B}(\mathbb{S}),\nu)$ is a probability measure, one easily sees that $P_S$ is a well-defined orthogonal projection on $\H$ satisfying 
\begin{align*}
P_S  f \in L^2(e^{-V}\mathrm{d}x)~\mbox{ and }~\| P_S  f\|_{L^2(e^{-V}\mathrm{d}x)}=\| P_S f\|_{\H},~f \in \H.
\end{align*}
In addition, assume that $e^{-V} \mathrm{d}x$ is a probability measure on $(\mathbb{R}^d,\mathcal{B}(\mathbb{R}^d))$ and define $L^2(e^{-V}\mathrm{d}x)$ to be $L^2(\mathbb{R}^d,e^{-V}\mathrm{d}x)$. Then $P:\H \rightarrow \H$ is given as
\begin{align*}
P f= P_S   f - \left(f,1\right)_\H,~f \in \H.
\end{align*}
It is again easily verified that $P$ is indeed an orthogonal projection and it holds 
\begin{align*}
P  f \in L^2(e^{-V}\mathrm{d}x)~\mbox{ and }~\| P  f\|_{L^2(e^{-V}\mathrm{d}x)}=\| P f\|_{\H},~f \in \H.
\end{align*}
Finally, for notational convenience, we write
\begin{align*}
f_S:=P_S f,~f \in \H.
\end{align*}
\end{Df}

The regularity properties of $P$ required in (D5) are contained in the following lemma.

\begin{Lm} \label{Lm_verification_D5_spherical_velocity_Langevin}
Let $V:\mathbb{R}^d \rightarrow \mathbb{R}$ be locally Lipschitz continuous and let $e^{-V} \mathrm{d}x$ be a probability measure on $(\mathbb{R}^d,\mathcal{B}(\mathbb{R}^d))$. Then it holds $P(\H) \subset D(S)$, $S P=0$ as well as $P(\D) \subset D(A)$ and $AP (\D) \subset D(A)$. Moreover, we have the following formula
\begin{align} \label{Formula_AP}
APf  = - \omega \cdot \nabla_x \,f_S,~f\in \D.
\end{align}
\end{Lm}

\begin{proof}
Note that the range of $P$ may naturally be identified with $L^2(e^{-V}\mathrm{d}x)$. So choose $f \in L^2(e^{-V}\mathrm{d}x)$. We have that $C_c^\infty(\mathbb{R}^d)$ is dense in $L^2(e^{-V}\mathrm{d}x)$. Thus there exists $f_n \in C_c^\infty(\mathbb{R}^d)$, $n \in \mathbb{N}$, such that $f_n \rightarrow f$ in $L^2(e^{-V}\mathrm{d}x)$ as $n \rightarrow \infty$. Now identify all $f_n$, $n \in \mathbb{N}$, and $f$ again with elements from $\H$. Then $f_n \in \D$ and $Sf_n=0$ for each $n \in \mathbb{N}$ as well as $f_n \rightarrow f$ in $\H$ as $n \rightarrow \infty$. And since $(S,D(S))$ is closed, this shows $f \in D(S)$ and $Sf=0$.

To show the second part, let $f \in \D$. Choose an open ball $B_r(0)$ around $0$ with radius $r \in (0,\infty)$ large enough such that $\text{supp}(f) \subset B_r(0) \times \mathbb{S}$. The support of $f_S$ is then contained in $B_r(0)$ and hence $f_S \in C^\infty_c(\mathbb{R}^d)$. Identifying any element of $C^\infty_c(\mathbb{R}^d)$ again with an element from $\D$ we conclude
\begin{align*}
f_S \in \D ~\mbox{ and }~A P_S f \,(x,\omega) = - \omega \cdot \nabla_x \,f_S \,(x),~(x,\omega) \in \mathbb{M}.
\end{align*}
Next we show that $1 \in D(A)$ and $A1=0$. Therefore, let $\varphi \in C_c^\infty(\mathbb{R}^d)$ be a cut-off function with $0 \leq \varphi \leq 1$ such that $\varphi=1$ on $B_1(0)$ and $\varphi =0$ outside $B_2(0)$. Define $\varphi_n$ as $\varphi_n(x)=\varphi(\frac{1}{n} x)$ for each $x \in \mathbb{R}^d$ and all $n \in \mathbb{N}$. Note that 
\begin{align*}
\left| \nabla \varphi_n(x) \right| \leq \frac{1}{n} \, C,~x \in \mathbb{R}^d,~n \in \mathbb{N},
\end{align*}
where $C=\sup_{ y \in \mathbb{R}^d } |\nabla \varphi (y) | < \infty.$ Canonically, we have $\varphi_n \in \D$ for all $n \in \mathbb{N}$ and by the construction of $(\varphi_n)_{n \in \mathbb{N}}$ we get
\begin{align*}
 A\varphi_n = - \omega \cdot \nabla_x \varphi_n \rightarrow  0 ~\mbox{ as }~n \rightarrow \infty
\end{align*}
in $\H$ by using dominated convergence. Clearly, also $\varphi_n \rightarrow 1$ in $\H$ as $n \rightarrow \infty$. Hence, as desired, $1 \in D(A)$ and $A1=0$ since $(A,D(A))$ is closed. Summarizing, we obtain $Pf \in D(A)$ and the representation
\begin{align*}
APf  = - \omega \cdot \nabla_x \,f_S
\end{align*}
for each $f \in \D$, in particular, $APf \in \D \subset D(A)$. This completes the proof.
\end{proof}

Next, we prove the conservativity condition in (D7) which is now almost obvious by the previous arguments.

\begin{Lm}
Let $V:\mathbb{R}^d \rightarrow \mathbb{R}$ be locally Lipschitz continuous and assume that $e^{-V} \mathrm{d}x$ is a probability measure on $(\mathbb{R}^d,\mathcal{B}(\mathbb{R}^d))$. It holds $1 \in D(L)$ and $L1=0$.
\end{Lm}

\begin{proof}
Let $(\varphi_n)_{n \in \mathbb{N}}$ be as in the proof of Lemma \ref{Lm_verification_D5_spherical_velocity_Langevin}. Note that $L\varphi_n=A\varphi_n$ for each $n \in \mathbb{N}$. Hence $\varphi_n \rightarrow 1$ and $L \varphi_n \rightarrow 0$ in $\H$ as $n \rightarrow \infty$. The claim follows since $(L,D(L))$ is closed.
\end{proof}

Summarizing, conditions (D1) and (D3)-(D7) are fulfilled. The hardest part is to show that $(L,D(L))$ generates a $C_0$-semigroup. This is done in Section \ref{Section_Construction_Semigroup} separately, see Theorem \ref{Thm_Essential_mdissipativity_smoothcase} and Theorem \ref{Thm_essential_mdissipativity_sphericalvelocityLangevin_locLipschitz} therein.

\subsection{The conditions (H1)-(H4)} \label{Section_(H1)_(H4)} In this section we verify Assumptions (H1)-(H4). The calculations are similar as the algebraic ones from \cite[Sec.~6]{GKMS12}. But first we introduce the necessary conditions on the potential $V$ needed later on. For criteria on Poincar\'e inequalities, consider e.g.~\cite{BBCG08} or \cite[A.~19]{Vil09}.
\begin{Ass}
$ $
\begin{itemize}
\item[(C1)] The potential $V:\mathbb{R}^d \rightarrow \mathbb{R}$ is bounded from below, satisfies $V \in C^{2}(\mathbb{R}^d)$ and $e^{-V} \mathrm{d}x$ is a probability measure on $(\mathbb{R}^d,\mathcal{B}(\mathbb{R}^d))$.
\item[(C2)] The probability measure $e^{-V}\mathrm{d}x$ satisfies a Poincar\'e inequality of the form
\begin{align*}
\left\|\nabla f \right\|^2_{L^2(e^{-V}\mathrm{d}x)} \geq \Lambda  \, \left\| f - \left(f,1\right)_{L^2(e^{-V}\mathrm{d}x)} \,\right\|^2_{L^2(e^{-V}\mathrm{d}x)}
\end{align*}
for some $\Lambda \in (0,\infty)$ and all $f \in C_c^\infty(\mathbb{R}^d)$.
\item[(C3)] There exists some constant $c < \infty$ such that
\begin{align*}
\left| \nabla^2 V (x) \right| \leq c \left( 1+ \left| \nabla V(x) \right|\right) \mbox{ for all } x \in \mathbb{R}^d.
\end{align*}
\end{itemize}
\end{Ass}
Here $|T|$ denotes the Frobenius norm of some real-valued matrix $T$. We note that (C2) is necessary to show (H3) and (C2) together with (C3) are indispensable to prove (H4). First we start with the verification of (H1). 

\begin{Pp} Let $V:\mathbb{R}^d \rightarrow \mathbb{R}$ be locally Lipschitz continuous and let $e^{-V} \mathrm{d}x$ be a probability measure on $(\mathbb{R}^d,\mathcal{B}(\mathbb{R}^d))$. Then (H1) holds.
\end{Pp} 

\begin{proof}
Let $f \in \D$. Then by Formula \eqref{Formula_AP} we have
\begin{align*}
A P f= - \omega \cdot \nabla_x f_S~\mbox{ where }~f_S=P_S f =\int_\mathbb{S} f \, \mathrm{d}\nu.
\end{align*}
By the Gaussian integral formula, see Lemma \ref{Gauss_formula_sphere}, we conclude $P_S A Pf =0$. Thus also $\left(A Pf,1\right)_\H=\left(P_S A Pf,1\right)_{L^2(e^{-V} \mathrm{d}x)}=0$. Hence $P A P = 0$ on $\D$.
\end{proof}

\begin{Pp}
Let $V:\mathbb{R}^d \rightarrow \mathbb{R}$ be locally Lipschitz continuous. Then condition (H2) holds with $\Lambda_m=\frac{1}{2} \sigma^2 (d-1)$.
\end{Pp}

\begin{proof}
This just follows by the Poincar\'e inequality on $\mathbb{S}$ (as in \cite[Sec.~6]{GKMS12} or \cite{DKMS11} respectively). Indeed, for all $f \in C^\infty(\mathbb{S})$ we have
\begin{align*}
\frac{1}{d-1} \int_{\mathbb{S}}  \left( \text{grad}_{\mathbb{S}} f, \text{grad}_{\mathbb{S}} f \right)_{T\mathbb{S}}  \, \mathrm{d}\nu \geq \int_{\mathbb{S}}  f^2 \, \mathrm{d}\nu - \left( \int_{\mathbb{S}} f \, \mathrm{d}\nu \right)^2,  
\end{align*}
see \cite[Theo.~2]{Be89}. This implies
\begin{align*} 
- \left( S f, f \right)_{\H} \geq \frac{1}{2} \sigma^2 (d-1)  \left\| (I-P_S)f \right\|_{\H}^2,~f \in \D.
\end{align*}
\end{proof}

In order to prove (H3), we aim to apply Corollary \ref{Cor_H3_equivalent}. Especially, we have to show  that $(I-PA^2P)(\D)$ is dense in $\H$. Before proving this, we first calculate the representation of $I-PA^2P$ on $\D$. The next proposition is satisfied in particular under the assumption that $V$ fulfills (C1). Let us therefore already assume it. By using the formulas from Section \ref{Section_Formulas_Notations} we get
\begin{align} \label{formula_AAP}
A^2 P f &= - A \,\left( \omega, \nabla f_S\right)_{\text{euc}} = \left( \omega, \nabla^2 f_S \, \omega \right)_{\text{euc}} -  \frac{1}{d-1} \,\left( \left(I-\omega\otimes \omega\right) \nabla V, \,\nabla f_S \right)_{\text{euc}}.
\end{align}
Hence again by the Gaussian integral formula, see Lemma \ref{Gauss_formula_sphere}, it follows easily (see also the corresponding computation in \cite[Sec.~6]{GKMS12} and in \cite{DKMS11}) that
\begin{align*} 
P_S \,A^2 P f = \frac{1}{d} \,\Delta f_S - \frac{1}{d}  \,\nabla V \cdot \nabla f_S.
\end{align*}
For the moment, consider the operator $(G,C_c^\infty(\mathbb{R}^d))$ defined by $G=\frac{1}{d} \, \Delta -  \frac{1}{d}\,\nabla V \cdot \nabla$ on $C_c^\infty(\mathbb{R}^d)$. Then for each $h \in C^\infty_c(\mathbb{R}^d)$ and $g \in C^\infty(\mathbb{R}^d)$ it holds
\begin{align*}
\left( Gh, g \right)_{L^2(e^{-V} \mathrm{d}x)} = - \frac{1}{d} \, \int_{\mathbb{R}^d} \nabla h \cdot \nabla g ~e^{-V} \mathrm{d}x.
\end{align*}
In particular, we have $\left( Gh, 1\right)_{L^2(e^{-V} \mathrm{d}x)}=0$. Thus, since $f_S \in C^\infty_c(\mathbb{R}^d)$, we conclude
\begin{align*}
\left(A^2 P f, 1\right)_\H= \left(Gf_S,1 \right)_{L^2(e^{-V} \mathrm{d}x)}=0.
\end{align*}
Altogether, for each $f \in \D$, we obtain the formula
\begin{align} \label{formula_PAAP}
P A^2 P f = \frac{1}{d} \, \Delta f_S - \frac{1}{d} \, \nabla V \cdot \nabla f_S.
\end{align}
Consequently, we get
\begin{align*}
(I- P A^2 P) f =  f - \frac{1}{d} \,\Delta f_S + \frac{1}{d}  \,\nabla V \cdot \nabla f_S ~~\mbox{ for all }f \in \D.
\end{align*}
Now we prove the desired upcoming proposition.

\begin{Pp} \label{density_of_PAAPD}
Let $V:\mathbb{R}^d \rightarrow \mathbb{R}$ satisfy (C1). Then $(I-P A^2 P)(\D)$ is dense in $\H$. In other words, $(I-PA^2P,\D)$ is essentially m-dissipative on $\H$.
\end{Pp}

\begin{proof}
Let $(G,C_c^\infty(\mathbb{R}^d))$ be as defined above. By \cite[Theo.~7]{BKR97}, it follows that $(G,C_c^\infty(\mathbb{R}^d))$ is essentially selfadjoint on $L^2(e^{-V} \mathrm{d}x)$. Hence $(G,C_c^\infty(\mathbb{R}^d))$ is also essentially m-dissipative on $L^2(e^{-V} \mathrm{d}x)$ (use e.g.~\cite[Ch.~1,~Cor.~4.4]{Paz83}). Now let $g \in \H$ such that
\begin{align} \label{eq_verification_core_H3}
((I-P A^2 P)f,g)_\H=0 ~~\mbox{ for all } f \in \D.
\end{align}
We have to show that $g=0$. Note that \eqref{eq_verification_core_H3} implies
\begin{align*} 
((I-G)f,g_S)_{L^2(e^{-V} \mathrm{d}x)}=0 ~~\mbox{ for all } f \in C_c^\infty(\mathbb{R}^d).
\end{align*}
Hence $g_S=0$ in $L^2(e^{-V} \mathrm{d}x)$ since $(I-G)(C_c^\infty(\mathbb{R}^d))$ is dense in $L^2(e^{-V} \mathrm{d}x)$. Thus for each $f \in \D$ we have 
\begin{align*}
(P A^2 Pf,g)_\H=\left(Gf_S,g_S\right)_{L^2(e^{-V} \mathrm{d}x)}=0.
\end{align*}
Consequently, \eqref{eq_verification_core_H3} yields $(f,g)_\H=0$ for all $f \in \D$. Hence $g=0$ as desired.
\end{proof}

Now we can prove (H3).

\begin{Pp} Assume that $V:\mathbb{R}^d \rightarrow \mathbb{R}$ satisfies (C1) and (C2). Then (H3) holds with $\Lambda_M=\frac{\Lambda}{d}$.
\end{Pp}

\begin{proof}
Let $f \in \D$. By the Gaussian integral formula, see Lemma \eqref{Gauss_formula_sphere}, and the Poincar\'e inequality for the probability measure $e^{-V}\mathrm{d}x$ in (C2) we obtain
\begin{align*}
\| A P f \|^2 &= \int_{\mathbb{R}^d} \int_{\mathbb{S}} \left( \omega \cdot \nabla_x f_S \right)^2 \,e^{-V} \mathrm{d}\nu(\omega) \,\mathrm{d}x = \frac{1}{d} \int_{\mathbb{R}^d} \left| \nabla_x f_S \right|^2 \, e^{-V} \mathrm{d}x \\
& \geq \frac{\Lambda}{d} \int_{\mathbb{R}^d} \left(f_S - \int f_S \,e^{-V}  \mathrm{d}x \right)^2 e^{-V} \mathrm{d}x = \frac{\Lambda}{d} \,\left\| P_S f - \left(f,1\right)_\H\right\|^2 .
\end{align*}
Thus Inequality \eqref{eq_inequality_D3} is satisfied for all elements from $\D$. By Corollary \ref{Cor_H3_equivalent} in combination with Proposition \ref{density_of_PAAPD} the claim follows.
\end{proof}

Finally, we prove (H4). This requires the usage of an elliptic regularity result from Dolbeault, Mouhot and Schmeiser (see the Appendix below) and needs especially Conditions (C1)-(C3) from above.

\begin{Pp} Assume that $V:\mathbb{R}^d \rightarrow \mathbb{R}$ satisfies (C1),(C2) and (C3). Then (H4) holds with $N_1=(d-1) \frac{\sigma^2}{4}$ and $N_2$  is some finite positive constant depending only on the potential $V$.
\end{Pp}

\begin{proof}
Now, in order to prove (H4), we aim to apply Lemma \ref{Pp_suff_H4_part1} and Proposition \ref{Pp_suff_H4_part2}. Therefore, let $f \in \D$. By Lemma \ref{Lm_phi_sphere} we conclude
\begin{align*}
P_S Af &= - \nabla_x \cdot \int_{\mathbb{S}} \omega \, f\, \mathrm{d}\nu(\omega) + \int_{\mathbb{S}} \,\text{grad}_{\mathbb{S}} \Phi(\omega) \cdot \nabla_\omega f \,\mathrm{d}\nu(\omega) \\
&= \Big( - \nabla_x + \nabla_x V \Big) \cdot \int_{\mathbb{S}} \omega \,f \, \mathrm{d}\nu(\omega).
\end{align*}
Note that $S(\D) \subset \D$. The last identity together with Lemma \ref{Lm_formula_sphere_strong_feller} yields 
\begin{align*}
P_S A \,Sf = -(d-1) \frac{\sigma^2}{2} \,P_S A f~~\mbox{ for all } f \in \D.
\end{align*} 
By the proof of Lemma \ref{Lm_properties_of_L}, we have $\left(Ah,1\right)_\H=0$ for all $h \in \D$. Using this, we get $PAh = P_S A h$ for all such $h$. Hence we obtain
\begin{align*}
P A \,S = -(d-1) \frac{\sigma^2}{2} \,P A ~~\mbox{ on } \D.
\end{align*} 
So the first part of (H4) is satisfied with $N_1=(d-1) \frac{\sigma^2}{4}$ due to Lemma \ref{Pp_suff_H4_part1}. Next, we show the second part of (H4). Therefore, let $g \in \H$ be of the form $g=(I-PA^2P)f$ for some $f \in \D$. Formula \eqref{g_in_DBA*} together with \eqref{formula_AAP} imply
\begin{align} \label{estimation_BA*_1}
\|(BA)^*g\|_{\H}  &\leq \| |\nabla_x^2  f_S| \|_{L^2(e^{-V}\mathrm{d}x)} + \frac{1}{d-1} \| |\nabla_x V| |\nabla_x f_S| \|_{L^2(e^{-V}\mathrm{d}x)}. 
\end{align}
Now due to Identity \eqref{formula_PAAP} note that $Pf=f_S-\left(f_S,1\right)_{L^2(e^{-V} \mathrm{d}x)}$ with $f_S \in C_c^\infty(\mathbb{R}^d)$ solves the equation
\begin{align*}
Pf - \frac{1}{d} \left( \Delta Pf - \nabla V \cdot \nabla Pf \right)=Pg ~~\mbox{ in } L^2(e^{-V} \mathrm{d}x).
\end{align*}
By applying the a priori estimates from Dolbeault, Mouhot and Schmeiser (see Proposition \ref{Elliptic_regularity_Prop} and Lemma \ref{Lm_Appendix_3} from the Appendix below) to \eqref{estimation_BA*_1}, which require Conditions (C1)-(C3), we conclude
\begin{align*} 
\|(BA)^*g\|_{\H}  &\leq N_2 \, \| Pg\|_{L^2(e^{-V} \mathrm{d}x)} \leq N_2 \,\|g\|_\H
\end{align*}
for some $N_2 < \infty$ independent of $g$. Now, finally, apply Proposition \ref{Pp_suff_H4_part2}.
\end{proof}

Collecting all results from the whole section (for the proof of (D2) see the upcoming section), Theorem \ref{Thm_Hypocoercivity} implies the final hypocoercivity theorem for the spherical velocity Langevin process stated already in the introduction, see Theorem \ref{Hypocoercivity_theorem_spherical_velocity_Langevin}. 

\section{Essential m-dissipativity of the spherical velocity Langevin generator} \label{Section_Construction_Semigroup}

In this section, we prove the that the closure of the spherical velocity Langevin generator $(L,C_c^\infty(\mathbb{M}))$ in $L^2(\mathbb{M},\mu)=L^2(\mu)$ generates a $C_0$-contraction semigroup under the assumption that $V \in C^\infty(\mathbb{R}^d)$ or that $V:\mathbb{R}^d \rightarrow \mathbb{R}$ is locally Lipschitz continuous and bounded from below. Here we follow the notations introduced in Definition \ref{Df_operator_weak_ass_potential} and in Section \ref{Section_Formulas_Notations}. We always assume that $\sigma>0$ without further mention this again. We start with the first case which makes use of methods developed in \cite{HN05}.

\subsection{The smooth case}
At first, we require the following lemma which shows that the vector fields involved in the spherical velocity Langevin operator are satisfying H"{o}rmander's condition. Below $L(X_1,\ldots,X_r)$ denotes the Lie-algebra generated by the smooth vector fields $X_1,\ldots,X_r$, $r \in \mathbb{N}$, on some smooth manifold $\mathbb{X}$, this is,  the least $\mathbb{R}$-vector space including all $X_i$, $i=1,\ldots,r,$ which is closed under the Lie-bracket operation. Moreover, define
\begin{align*}
L(X_1,\ldots,X_r)(p)=\text{span} \big\{ X(p) ~\big|~X \in L(X_1,\ldots,X_r) \big\} \subset T_p\mathbb{X}
\end{align*}
where $T_p\mathbb{X}$ is the tangent space of $\mathbb{X}$ at the point $p \in \mathbb{X}$. The statement of the following lemma can of course also be obtained using local coordinates. However, the upcoming coordinate free proof is maybe more direct. It is similar to the proof of Lemma 5.1 in \cite{GKMS12}.

\begin{Lm} \label{Lm_Hypoellipticity_L}
Let $V \in C^\infty(\mathbb{R}^d)$ and let $\mathcal{A}$, $\mathcal{S}_1,\ldots,\mathcal{S}_d$ be the smooth vector fields living on $\mathbb{M}$ introduced in \eqref{Def_Sn} and in Definition \ref{Df_operator_weak_ass_potential}. Then we have
\begin{align*}
\mbox{dim} ~L\big(\mathcal{A},\mathcal{S}_1,\ldots,\mathcal{S}_d\big)(p)=2d-1 ~~\mbox{ at each point $p \in \mathbb{M}$}.
\end{align*}
\end{Lm}

\begin{proof}
Choose an arbitrary $p=(x,\omega) \in \mathbb{M}$. First recall 
\begin{align*}
\text{span} \{\omega\} \oplus \text{span} \{ (I-\omega \otimes \omega)e_n~|~n=1,\ldots,d\, \} = \text{span} \{\omega\} \oplus T_\omega\mathbb{S} =\mathbb{R}^d,
\end{align*} 
where $e_n$ is the $n$-th unit vector in $\mathbb{R}^d$. Then note that $\left[\mathcal{S}_n,\mathcal{A}\right]$, $n=1,\ldots,d,$ is of the form
\begin{align*}
\mathcal{N}_n=\left[\mathcal{S}_n,\mathcal{A}\right] = (I-\omega \otimes \omega) e_n \cdot \nabla_x + f^{\,(n)}(x,\omega) \cdot \nabla_\omega 
\end{align*}
for some smooth functions $f^{\,(n)}: \mathbb{M} \rightarrow \mathbb{R}^d$ which are tangential to $\mathbb{S}$. Hence under
\begin{align*}
\mathcal{A}(p),~\mathcal{S}_1(p),\ldots,~ \mathcal{S}_d(p),~\mathcal{N}_1(p),\ldots,~\mathcal{N}_d(p),
\end{align*}
we may always choose $2d-1$-linear independent vectors. The claim follows since $T_p\mathbb{M}$ is of dimension $2d-1$.
\end{proof}

We are arriving at our first core result of this section. The idea for the proof we learned from \cite[Prop.~5.5]{HN05}. Moreover, we remark that the arguments for showing essential m-dissipativity of the classical Langevin generator (or kinetic Fokker-Planck operator respectively) from \cite[Prop.~5.5]{HN05} based on hypoellipticity have already been detected some time before in the surprisingly, seemingly unknown article from Soloveitchik, see \cite{Sol95} and Lemma 3.8 therein.

\begin{Thm} \label{Thm_Essential_mdissipativity_smoothcase}
Assume that $V \in C^\infty(\mathbb{R}^d)$. Then $(L,C_c^\infty(\mathbb{M}))$ is essentially m-dissipative on $L^2(\mathbb{M},\mu)$. Thus its closure $(L,D(L))$ generates a $C_0$-contraction semigroup $(T(t))_{t \geq 0}$ on $L^2(\mathbb{M},\mu)$.
\end{Thm}

\begin{proof}
First of all, dissipativity of $(L,C_c^\infty(\mathbb{M}))$ on $L^2(\mathbb{M},\mu)$ is clear by Lemma \ref{Lm_properties_of_L}. By the Lumer-Phillips theorem, see \cite{Gol85} or \cite{LP61}, it suffices to verify that $(I-L)(C_c^\infty(\mathbb{M}))$ is dense in $L^2(\mathbb{M},\mu)$. Therefore, let $f \in L^2(\mathbb{M},\mu)$ such that
\begin{align} \label{eq_proof_Thm_Essential_mdissipativity_smoothcase}
\left( (I-L)u,f \right)_{L^2(\mu)} = 0 ~~\mbox{ for all } u \in C_c^\infty(\mathbb{M}).
\end{align}
We have to show that $f=0$. Note that $e^{-V}f \in L^1_{\text{loc}}(\mathbb{M},\mathrm{d}x \otimes \nu)$. From Lemma \ref{Lm_Hypoellipticity_L} and H"{o}rmander's hypoellipticity theorem we can infer that $e^{-V}f \in C^\infty(\mathbb{M})$, see Proposition \ref{Pp_consequence_Hormander_hypoellipticity} in the Appendix below. Hence $f \in C^\infty(\mathbb{M})$. Now choose cut-off functions $\varphi_n \in C_c^\infty(\mathbb{R}^d)$, $0 \leq \varphi_n \leq 1$, $n \in \mathbb{N}$, be as defined in the proof of Lemma \ref{Lm_verification_D5_spherical_velocity_Langevin}. Recall that there exists $C< \infty$ such that
\begin{align*}
\sup_{x \in \mathbb{R}^d} \left| \nabla \varphi_n(x) \right| \leq \frac{1}{n} \, C ~\mbox{ and }~\varphi_n \rightarrow 1 \mbox{ pointwise as } n \rightarrow \infty.
\end{align*}
Let $n \in \mathbb{N}$ and set $u_n=\varphi_n^2 f$ (as in the proof of Proposition 5.5 in \cite{HN05}). Then \eqref{eq_proof_Thm_Essential_mdissipativity_smoothcase} yields
\begin{align*}
\left( u_n,f \right)_{L^2(\mu)} =  \left( L u_n,f \right)_{L^2(\mu)} = \left( S u_n,f \right)_{L^2(\mu)} - \left( A u_n,f \right)_{L^2(\mu)}.
\end{align*}
Since $(S,C_c^\infty(\mathbb{M}))$ is nonpositive definite, we get
\begin{align*}
\left( S u_n,f \right)_{L^2(\mu)}=\left( S (\varphi_n f),\varphi_n f \right)_{L^2(\mu)} \leq 0
\end{align*}
We have $Au_n=(A\varphi_n) \, \varphi_n f  \,+ \,\varphi_n\, A(\varphi_nf)$. So by the antisymmetry of $(A,C_c^\infty(\mathbb{M}))$,
\begin{align*}
\left( A u_n,f \right)_{L^2(\mu)}= \left( (A \varphi_n) \varphi_n f, f \right)_{L^2(\mu)} = - \int_\mathbb{M} \left(\omega \cdot \nabla \varphi_n \right) \,\varphi_n \,f^2\,e^{-V} \mathrm{d}x \otimes \nu.
\end{align*}
Altogether, we obtain
\begin{align*}
\int_\mathbb{M} \varphi_n^2 \, f^2\, e^{-V} \mathrm{d}x \otimes \nu  \leq \frac{1}{n} \, C \,\int_\mathbb{M} \varphi_n \,f^2\,e^{-V} \mathrm{d}x \otimes \nu \leq \frac{1}{n} \, C\, \|f\|^2.
\end{align*}
So for $n \rightarrow \infty$, we conclude $\|f\|^2 \leq 0$ by dominated convergence. Hence $f=0$.
\end{proof}

In Subsection \ref{Section_locallyLipschitzContinuousCase}, we extend the statement of the previous theorem to the setting in which $V$ is locally Lipschitz continuous and bounded from below. The methods for this have been developed in \cite[Sec.~4]{CG08} and \cite[Sec.~2]{CG10} and can also be applied to the spherical velocity Langevin generator in a similar way, see next.

\subsection{The locally Lipschitz continuous case} \label{Section_locallyLipschitzContinuousCase}

So let us proceed similar as in \cite[Sec.~4]{CG08}, or \cite[Sec.~2]{CG08} respectively. Therefore, we need some lemmas first. We set
\begin{align*}
D_1=C_c^\infty(\mathbb{R}^d) \otimes C^\infty(\mathbb{S})=\text{span}\big\{ f \otimes g ~\big|~ f \in C^{\infty}_c(\mathbb{R}^d),~g \in C^\infty(\mathbb{S}) \big\}
\end{align*}
where $f \otimes g$ is defined by $(f \otimes g)(x,\omega)=f(x)\,g(\omega)$ for $(x,\omega) \in \mathbb{M}$ therein and let $(L_0,D_1)$ be given via
\begin{align*}
L_{0}= \omega \cdot \nabla_x + \frac{\sigma^2}{2}\,\Delta_\mathbb{S} ~~\mbox { on } D_1.
\end{align*} 
Then the following statement holds.

\begin{Lm} \label{Lm_ess_mdiss_localLip_1}
$(L_0,D_1)$ is essentially m-dissipative on $L^2(\mathbb{M},\mathrm{d}x \otimes \nu)$.
\end{Lm}

\begin{proof}
Note that Theorem \ref{Thm_Essential_mdissipativity_smoothcase} implies that $(L_0,C_c^\infty(\mathbb{M}))$ is essentially m-dissipative on $L^2(\mathbb{M},\mathrm{d}x \otimes \nu)$. We have to show that $(L_0,C_c^\infty(\mathbb{M}))$ is contained in the closure of $(L_0,D_1)$. Let $f \in C_c^\infty(\mathbb{M})$ and choose an extension $\widetilde{f} \in C_c^\infty(\mathbb{R}^{2d})$ of $f$. By using e.g.~the representation of $\Delta_\mathbb{S}$ in terms of the vector fields introduced in \eqref{Def_Sn}, $L_0$ can be extended to some smooth second order differential operator $\widetilde{L_0}$ on $\mathbb{R}^{2d}$ satisfying
\begin{align*}
\widetilde{L_0} \,\widetilde{h}_{\,|\mathbb{M}} = L_0h ~\mbox{ for all } ~h \in C^\infty(\mathbb{M}),~\widetilde{h} \in C^\infty(\mathbb{R}^{2d})~ \mbox{ with }~\widetilde{h}_{\,|\mathbb{M}}=h.
\end{align*}  
Now there exists $\widetilde{f_n} \in C^\infty_c(\mathbb{R}^d) \otimes C^\infty_c(\mathbb{R}^d)$, $n \in \mathbb{N}$, converging to $\widetilde{f}$ w.r.t.~the topology induced by the locally convex vector space $\mathcal{D}(\mathbb{R}^{2d})=C^\infty_c(\mathbb{R}^{2d})$, for instance see \cite[Ch.~4, Sec.~8,~Prop.~1]{Hor66}. In particular, all $\text{supp}(\widetilde{f_n})$ and $\text{supp}(\widetilde{f})$ are contained in some common compact set of $\mathbb{R}^{2d}$ and we easily see that
\begin{align*}
\sup_{(x,\omega)\in \mathbb{R}^{2d}} \left|\widetilde{L_0}\,\widetilde{f_n}(x,\omega) - \widetilde{L_0}\,\widetilde{f}(x,\omega)\right| \rightarrow 0
\end{align*}
as $n \rightarrow \infty$. Hence also
\begin{align*}
\sup_{(x,\omega)\in \mathbb{M}} \left|L_0f_n(x,\omega) - L_0f(x,\omega)\right| \rightarrow 0,
\end{align*}
where $f_n \in D_1$ denotes the restriction of $\widetilde{f_n}$ to $\mathbb{M}$, $n \in \mathbb{N}$. Thus we conclude that $f_n \rightarrow f$ and $Lf_n \rightarrow Lf$ in $L^2(\mathbb{M},\mathrm{d}x \otimes \nu)$ as $n \rightarrow \infty$, implying the claim.
\end{proof}

Let $V:\mathbb{R}^d \rightarrow \mathbb{R}$ again be locally Lipschitz continuous and denote by $H^{1,\infty}_c(\mathbb{R}^d)$ the vector space of all $H^{1,\infty}(\mathbb{R}^d)$-functions that vanish outside some bounded set. Introduce 
\begin{align*}
D_2=H^{1,\infty}_c(\mathbb{R}^d) \otimes C^\infty(\mathbb{S}) = \text{span}\big\{ f \otimes g ~\big|~ f \in H^{1,\infty}_c(\mathbb{R}^d),~g \in C^\infty(\mathbb{S}) \big\}.
\end{align*}
By analyzing the proof of Proposition \ref{Lm_properties_of_L} one directly sees that $(L,D_2)$ (with $L=S-A$ on $D_2$ , $S$ as before and $A$ directly defined via \eqref{alternative_repr_A_smooth_case}) is still a well-defined linear operator on $L^2(\mathbb{M},\mu)$ having the properties from Proposition \ref{Lm_properties_of_L} with $D$ replaced by $D_2$ therein. Now consider the unitary isomorphism
\begin{align*}
U:L^2(\mathbb{M},\mu) \rightarrow L^2(\mathbb{M},\mathrm{d}x \otimes \nu),~f \mapsto Uf=e^{-\frac{V}{2}}f.
\end{align*}
Define $(\widetilde{L},\widetilde{D_2})$ to be the transformation of $(L,D_2)$ under $U$, this is, 
\begin{align*}
\widetilde{D_2}=U (D_2),~\widetilde{L}=ULU^{-1}~\mbox{ on } \widetilde{D_2}.
\end{align*}
Clearly, $\widetilde{D_2}=D_2$ and it is easy to check that $\widetilde{L}$ can be written on $D_2$ as 
\begin{align*}
\widetilde{L}=S-\widetilde{A}
\end{align*}
where $(S,D_2)$ is symmetric and nonpositive definite on $L^2(\mathbb{M},\mathrm{d}x \otimes \nu)$, $(\widetilde{A},D_2)$ is antisymmetric on $L^2(\mathbb{M},\mathrm{d}x \otimes \nu)$ and we have
\begin{align*}
\widetilde{A} = - \omega \cdot \nabla_x + \text{grad}_{\mathbb{S}}\Phi \cdot \nabla_\omega - \frac{1}{2} \, \nabla V \cdot \omega ~~\mbox{ on } D_2.
\end{align*}
We first show essential m-dissipativity of $(L,D_2)$ on $L^2(\mathbb{M},\mu)$ under the assumption that $V$ is globally Lipschitz continuous. By the previous considerations this property is equivalent to essential m-dissipativity of $(\widetilde{L},D_2)$ on $L^2(\mathbb{M},\mathrm{d}x \otimes \nu)$. Before proving the latter, we recall first the following pertubation theorem. For the proof, see \cite[Cor.~3.8,~Lem.~3.9,~Prob.~3.10]{Dav80}.

\begin{Thm} \label{Pertubation_Theorem}
Let $(K,\mathcal{D})$ be an essentially m-dissipative operator on some Hilbert space $\mathcal{H}$ and let $(T,\mathcal{D})$ be dissipative. Assume that there exists finite constants $c_1 \in \mathbb{R}$ and $c_2 \geq 0$ such that 
\begin{align*}
\|Tf\|^2 \leq c_1 \left(Kf,f\right)_X + c_2 \,\|f\|^2 ~~\mbox{ for all } f \in \mathcal{D}.
\end{align*}
Then $(K+T,\mathcal{D})$ is essentially m-dissipative on $\mathcal{H}$.
\end{Thm}

As announced before, the following lemma holds. The idea for the proof is obtained from the proof of Lemma 7 in \cite{CG08}.

\begin{Lm} \label{Lm_ess_mdiss_localLip_2}
Assume that $V:\mathbb{R}^d \rightarrow \mathbb{R}$ is globally Lipschitz continuous. Then $(\widetilde{L},D_2)$ is essentially m-dissipative on $L^2(\mathbb{M},\mathrm{d}x \otimes \nu)$. Hence $(L,D_2)$ is essentially m-dissipative on $L^2(\mathbb{M},\mu)$.
\end{Lm}

\begin{proof}
Let $(K,D_2)$ be defined as $K=L_0$ on $D_2$. Then $(K,D_2)$ is a dissipative extension of $(L_0,D_1)$, hence it is essentially m-dissipative on $L^2(\mathbb{M},\mathrm{d}x \otimes \nu)$ by Lemma \ref{Lm_ess_mdiss_localLip_1}. Define $(T,D_2)$ as
\begin{align*}
T= -\text{grad}_{\mathbb{S}}\Phi \cdot \nabla_\omega + \frac{1}{2} \, \nabla V \cdot \omega ~~\mbox { on } D_2.
\end{align*}
Clearly, $(\widetilde{A},D_2)$ and $(\omega \cdot \nabla_x,D_2)$ are antisymmetric on $L^2(\mathbb{M},\mathrm{d}x \otimes \nu)$. Thus $(T,D_2)$ is antisymmetric on $L^2(\mathbb{M},\mathrm{d}x \otimes \nu)$, hence also dissipative. Now note that for $\phi:\mathbb{S} \rightarrow \mathbb{R}$ defined as in Lemma \ref{Lm_phi_sphere} and all $h \in C^\infty(\mathbb{S})$ we have
\begin{align*}
\text{grad}_\mathbb{S} \, \phi (h) = \left((I-\omega \otimes \omega) z , \nabla_w h\right)_{\text{euc}} = \left(z,\text{grad}_\mathbb{S}h\right)_\text{euc}
\end{align*}
So it holds
\begin{align*}
\left| \text{grad}_\mathbb{S} \, \phi (h) \right|^2 \leq |z|^2 \left(\text{grad}_\mathbb{S}h, \text{grad}_\mathbb{S}h\right)_{T\mathbb{S}}.
\end{align*}
Having the latter inequality in mind and using Green's formula, we obtain for all $f \in D_2$ that
\begin{align*}
\|Tf\|^2 \leq 2 \,\| \text{grad}_\mathbb{S} \Phi (f)\|^2 + \frac{c^2}{2}\,\|f\|^2 &\leq 2 \,c^2 \left(-\Delta_\mathbb{S} f,f \right)_{L^2(\mathbb{M},\mathrm{d}x \otimes \nu)} + \frac{c^2}{2}\,\|f\|^2 \\
&=4 \,\frac{c^2}{\sigma^2} \left(-K f,f \right)_{L^2(\mathbb{M},\mathrm{d}x \otimes \nu)} + \frac{c^2}{2}\,\|f\|^2
\end{align*}
Here $c< \infty$ is the $L^\infty$-bound of $\nabla V$ on $\mathbb{R}^d$. So the claim follows by applying Kato pertubation, see Theorem \ref{Pertubation_Theorem}.
\end{proof}

Now we are arriving at the first intermediate result, similar to \cite[Theo.~2.1]{CG10}.

\begin{Pp} \label{Thm_essential_mdissipativity_sphericalvelocityLangevin_globLipschitz}
Assume that $V:\mathbb{R}^d \rightarrow \mathbb{R}$ is globally Lipschitz continuous. Then $(L,C_c^\infty(\mathbb{M}))$ is essentially m-dissipative on $L^2(\mathbb{M},\mu)$.
\end{Pp}

\begin{proof}
Since $(L,C_c^\infty(\mathbb{M}))$ is a dissipative extension of $(L,D_1)$, it suffices to prove that $(L,D_1)$ is essentially m-dissipative on $L^2(\mathbb{M},\mu)$. On the other hand, the latter statement follows by showing that $(L,D_2)$ is contained in the closure of $(L,D_1)$ and by using Lemma \ref{Lm_ess_mdiss_localLip_2} from above. Therefore, let $f \otimes g$ be a pure tensor with $f \in H^{1,\infty}_c(\mathbb{R}^d)$ and $g \in C^\infty(\mathbb{S})$. In particular, there exists $f_n \in C^\infty_c(\mathbb{R}^d)$, $n \in \mathbb{N}$, all with support contained in some common compact set such that $f_n \rightarrow f$ and $\partial_{x_i} f_n \rightarrow \partial_{x_i}f$ in $L^2(\mathbb{R}^d,\mathrm{d}x)$ as $n \rightarrow \infty$ where $i=1,\ldots,d$. Consequently, we have $f_n \rightarrow f$ and $\partial_{x_i} f_n \rightarrow \partial_{x_i}f$ even in $L^2(\mathbb{R}^d,e^{-V}\mathrm{d}x)$ as $n \rightarrow \infty$ for $i=1,\ldots,d$. This again easily implies that $f_n \otimes g \rightarrow f \otimes g$ and $L(f_n \otimes g) \rightarrow L(f \otimes g)$ in $L^2(\mathbb{M},\mu)$ as $n \rightarrow \infty$, finishing the proof.
\end{proof}

Finally, here is the core result of this section. As mentioned above, the strategy for its proof goes back to \cite[Cor.~2.3]{CG10}.

\begin{Thm} \label{Thm_essential_mdissipativity_sphericalvelocityLangevin_locLipschitz}
Assume that $V:\mathbb{R}^d \rightarrow \mathbb{R}$ is locally Lipschitz continuous and bounded from below. Then $(L,C_c^\infty(\mathbb{M}))$ is essentially m-dissipative on $L^2(\mathbb{M},\mu)$.
\end{Thm}

\begin{proof}
For notational convenience, we write $\mu_V$ and $L_V$ instead of $\mu$ and $L$ in this proof in order to indicate the dependence on the potential. W.l.o.g.~we assume that $V \geq 0$. Let $0 \not= g \in C_c^\infty(\mathbb{M})$ and let $\varepsilon >0$. Choose $\varphi \in C_c^\infty(\mathbb{R}^d)$ and $\psi \in C_c^\infty(\mathbb{R}^d)$ such that $\varphi = 1$ on $\text{supp}(g)$ and $\psi=1$ on $\text{supp}(\varphi)$ and $0 \leq \varphi \leq \psi \leq 1$. Furthermore, let $f \in C_c^\infty(\mathbb{M})$ at first be arbitrary. By using that $\psi V \leq V$ we get
\begin{align*}
&\left\|(I-L_V)(\varphi f) - g \right\|_{L^2(\mu_V)} \\
&\leq \left\|\varphi \left((I-L_{\psi V})f - g\right) \right\|_{L^2(\mu_{\psi V})} + \, \text{sup}_{x \in \mathbb{R}^d} \left| \nabla \varphi(x) \right| \|f\|_{L^2(\mu_{\psi V})} \\
&\leq \left\|(I-L_{\psi V})f - g \right\|_{L^2(\mu_{\psi V})} + \, \text{sup}_{x \in \mathbb{R}^d} \left| \nabla \varphi(x) \right| \, \left\|(I-L_{\psi V})f \right\|_{L^2(\mu_{\psi V})}.
\end{align*}
Here the last inequality is due to the dissipativity of $(L_{\psi V},C_c^\infty(\mathbb{M}))$ in $L^2(\mu_{\psi V})$. Now fix $\varphi$ such that 
\begin{align*}
\text{sup}_{x \in \mathbb{R}^d} \left| \nabla \varphi(x) \right| \leq \frac{\varepsilon}{4 ~\|g\|_{L^2(\mu_0)}}
\end{align*}
and let $\psi$ be as required above. For the construction of $\varphi$, see e.g.~the choice of the sequence $(\varphi_n)_{n \in \mathbb{N}}$ in the proof of Lemma \ref{Lm_verification_D5_spherical_velocity_Langevin}. Note that $\psi V$ is globally Lipschitz continuous, thus by Theorem \ref{Thm_essential_mdissipativity_sphericalvelocityLangevin_globLipschitz} and the Lumer-Phillips theorem, there exists $f \in C_c^\infty(\mathbb{M})$ satisfying
\begin{align*}
&\left\|(I-L_{\psi V})f - g \right\|_{L^2(\mu_{\psi V})} \leq \frac{\varepsilon}{2},~\left\|(I-L_{\psi V})f \right\|_{L^2(\mu_{\psi V})}\leq 2 \,\|g\|_{L^2(\mu_{\psi V})}.
\end{align*}
Altogether, we get $\left\|(I-L_V)(\varphi f) - g \right\|_{L^2(\mu_V)} \leq \varepsilon$. Hence $(I-L_V)(C_c^\infty(\mathbb{M}))$ is dense in $L^2(\mu_V)$ and the claim follows.
\end{proof}

\begin{Rm}
In order to prove Theorem \ref{Thm_essential_mdissipativity_sphericalvelocityLangevin_locLipschitz}, remember that we started with essential m-dissipativity of $L_0$ in Lemma \ref{Lm_ess_mdiss_localLip_1}. For the latter, the hypoellipticity statement of Theorem \ref{Thm_Essential_mdissipativity_smoothcase} was used. We remark that showing essential m-dissipativity of $L_0$ can alternatively be proven without using any hypoellipticity argument. This can easily be seen by using similar arguments as presented at the beginning of the proof of Theorem 2.1 in \cite{CG10}, or the beginning of Section 4 in \cite{CG08} respectively.
\end{Rm}

\section{Appendix}

Here we shortly discuss and recall some a priori results obtained by Dolbeault, Mouhot and Schmeiser in \cite{DMS10} as well as give a useful corollary as consequence of H"{o}rmander's hypoellipticity theorem. Both are needed Section \ref{Hypocoercivity_section_Spherical_velocity_Langevin} or Section \ref{Section_Construction_Semigroup} respectively. We start with the second one. This statement is well-known, at least in slightly different form, see \cite[Theo.~3]{IK74} or \cite[App.~A]{Cal12}. However, we have not found a precise reference. So we give the short proof below.

\begin{Pp} \label{Pp_consequence_Hormander_hypoellipticity}
Let $(\mathbb{M},g)$ be a smooth Riemannian manifold equipped with the induced Riemannian measure $\lambda_g$. Let $A$ be a differential operator of the form
\begin{align*}
A= \sum_{i=1}^r X_i^2 + X_0 + a
\end{align*}
where all $X_0,\ldots,X_r$, $r \in \mathbb{N}$, are smooth vector fields on $\mathbb{M}$ and $a \in C^\infty(\mathbb{M})$. Assume that $X_0,\ldots,X_r$ are satisfying H"{o}rmander's condition on $\mathbb{M}$, this is,
\begin{align*}
L(X_0,X_1,\ldots X_r)(p) = \text{dim}~T_p\mathbb{M} ~~\mbox{ at each point $p \in \mathbb{M}$}.
\end{align*}
Let $f \in L^1_{\text{loc}}(\mathbb{M},\lambda_g)$ such that
\begin{align*}
\int_\mathbb{M} A\psi \, f \, \mathrm{d}\lambda_g= 0 ~~\mbox{ for all } \psi \in C_c^\infty(\mathbb{M}).
\end{align*}
Then $f$ can already be represented by an element from $C^\infty(\mathbb{M})$.
\end{Pp}

\begin{proof} Let $(U,\varphi)$ be a coordinate chart of $\mathbb{M}$. We define the smooth vector fields $Y_0,\ldots,Y_r$ on the open subset $\varphi(U)$ in euclidean space as 
\begin{align*}
Y_{i}\left(h\right)(p)=X_i \left(h \circ \varphi\right)(\varphi^{-1} (p))
\end{align*}
where $h\in C^\infty(\varphi(U))$ and $p \in \varphi(U)$. Furthermore, let $b=a \,\circ\, \varphi^{-1}$ and define the second order differential operator $A_{\varphi}:C^\infty(\varphi(U)) \rightarrow C^\infty(\varphi(U))$ as
\begin{align*}
A_{\varphi}= \sum_{i=1}^r Y_i^2 + Y_0 + b.
\end{align*}
Consequently, $Y_0,\ldots,Y_r$ are satisfying H"{o}rmander's condition at every point of $\varphi(U)$. Now let $A_\varphi^{\,*}$ be the formal dual of $A_\varphi$ with respect to the Lebesgue measure $\mathrm{d}x$ on $\varphi(U)$. It can easily be shown that
\begin{align*}
A_\varphi^{\,*} = \sum_{i=1}^r Y_i^2 + Y_0' + c,
\end{align*}
where $c \in C^\infty(\mathbb{M})$ and $Y_0' + Y_0$ can be written as $\sum_{i=1}^r d_i Y_i$ for some $d_i \in C^\infty(\varphi(U))$, see e.g.~\cite[Eq.~(4.4)]{IK74}. Hence also $Y_0',Y_1,\ldots,Y_r$ are satisfying H"{o}rmander's condition on the whole of $\varphi(U)$. Now for $\psi \in C_c^\infty(\varphi(U))$ we have
\begin{align} \label{Eq_hypoellipticity_general}
\int_{\varphi(U)} A_\varphi \psi \left(f \circ \varphi^{-1}\right)\,\left(\text{vol}_g^{\varphi} \circ \varphi^{-1} \right)\,\mathrm{d}x = \int_{\mathbb{X}} A (\psi \circ \varphi) \,f\, \mathrm{d}\lambda_g =0
\end{align}
where $\text{vol}_g^{\varphi} \in C^\infty(U)$ denotes as usual the Riemannian volume form induced by the metric $g$ on $U$ through $\varphi$. Define $h$ as $h=\left(f \circ \varphi^{-1}\right)\,\left(\text{vol}_g^{\varphi} \circ \varphi^{-1} \right)$ on $\varphi(U)$. Then we get $h\in L^1_{\text{loc}}(\varphi(U),\mathrm{d}x)$ since $f \in L^1_{\text{loc}}(\mathbb{M},\lambda_g)$. So \eqref{Eq_hypoellipticity_general} implies
\begin{align*}
A_\varphi^{\,*}\,h=0,
\end{align*}
understood in the distributional sense on $\varphi(U)$. By H"{o}rmander's theorem we can infer that $h \in C^\infty(\varphi(U))$ since $A_\varphi^{\,*}$ is hypoelliptic, see \cite{Hor67}. Hence $f_{|U} \in C^\infty(U)$ since $\text{vol}_g^{\varphi}$ is strictly positive everywhere on $U$. Thus the claim follows.
\end{proof}

\subsection{Some a priori estimates}

In this subsection we recapitulate and discuss some a priori estimates obtained by Dolbeault, Mouhot and Schmeiser in \cite{DMS10} for a suitable elliptic equation which we need in order to prove hypocoercivity of our spherical velocity Langevin type process.  We assume that $V$ is an element from $C^2(\mathbb{R}^d)$, bounded from below and $e^{-V} \mathrm{d}x$ and is assumed to be a probability measure on $(\mathbb{R}^d,\mathcal{B}(\mathbb{R}^d))$. Here $d \in \mathbb{N}$ is fixed. The Hilbert space $L^2(\mathbb{R}^d,e^{-V}\mathrm{d}x)$ is also abbreviated by $L^2(e^{-V}\mathrm{d}x)$. The canonical scalar product is denoted by $\left( \cdot, \cdot \right)_{L^2(e^{-V}\mathrm{d}x)}$, the induced norm with $\| \cdot \|$. For notational convenience, we introduce the subspace $\mathcal{X}$ via
\begin{align*}
\mathcal{X}=\big\{ u \in L^2(e^{-V}\mathrm{d}x)~|~ u=f - \left(f,1\right)_{L^2(e^{-V}\mathrm{d}x)} \mbox{ for some } f \in C_c^\infty(\mathbb{R}^d) \big\}.
\end{align*}
The desired elliptic equation reads as follows. Assume that $u \in \mathcal{X}$ solves
\begin{align} \label{eq_Elliptic_Equation}
u - c_1 \left( \Delta u - \nabla V \cdot \nabla u \right) = g
\end{align}
in ${L^2(e^{-V}\mathrm{d}x)}$ for some $g \in {L^2(e^{-V}\mathrm{d}x)}$. Note that $\left( g,1 \right)_{L^2(e^{-V}\mathrm{d}x)}=0$ follows automatically. Here $c_1 \in (0,\infty)$ is a fixed constant. We need regularity estimates on $u$ and its first and second derivatives. So we may apply the a priori estimates derived in \cite[Sec.~2]{DMS10}. As mentioned before, domain issues are not considered in \cite{DMS10} and for completeness, we include slight modifications of some proofs in \cite[Sec.~2]{DMS10} below, mainly, in order to guarantee some integration by parts formulas and integrability requirements used therein. Our relevant case, $u \in \mathcal{X}$, causes no deep difficulties. The conditions on $V$ are adapted from \cite[Sec.~2]{DMS10} (or \cite{DMS09} respectively) and read now in similar form as follows.

\begin{Ass}
$ $
\begin{itemize}
\item[(A1)] The potential $V:\mathbb{R}^d \rightarrow \mathbb{R}$ is bounded from below, satisfies $V \in C^{2}(\mathbb{R}^d)$ and $ e^{-V} \mathrm{d}x$ is a probability measure on $(\mathbb{R}^d,\mathcal{B}(\mathbb{R}^d))$.
\item[(A2)] The probability measure $e^{-V}\mathrm{d}x$ satisfies a Poincar\'e inequality of the form
\begin{align*}
\|\nabla f \|^2 \geq \Lambda_{1}  \, \| f - \left(f,1\right)_{L^2(e^{-V}\mathrm{d}x)} \|^2,~ f \in C_c^\infty(\mathbb{R}^d),
\end{align*}
for some $\Lambda_{1} \in (0,\infty)$.
\item[(A3)] There exists $c_2 \in (0,\infty)$ and $c_3 \in [0,\frac{1}{2})$ such that
\begin{align*}
\Delta V \leq c_2 + c_3 \left| \nabla V \right|^2.
\end{align*}
\item[(A4)] It holds $| \nabla V| \in L^2(e^{-V}\mathrm{d}x)$.
\item[(A5)] There exists $c_4 \in (0,\infty)$ such that
\begin{align*}
\left| \nabla W \right| \leq c_4 \left( 1 + \left| \nabla V \right| \right) ,~W:=\sqrt{1+ \left| \nabla V \right|^2}.
\end{align*}\end{itemize}
\end{Ass}

Now we are arriving at the desired lemmas and results from \cite[Sec.~2]{DMS10}, only slightly reformulated.

\begin{Lm} \label{Lm_Appendix_1}
Assume that $(A1)-(A4)$ holds. Then there exists $\Lambda_{2} \in (0,\infty)$ such that for all $u=f- \left(f,1\right)_{L^2(e^{-V}\mathrm{d}x)}$, $f \in C^1_c(\mathbb{R}^d)$, we have
\begin{align*}
\| \nabla u \|^2 \geq \Lambda_{2} \, \| u \, \nabla V \|^2.
\end{align*}
\end{Lm}

\begin{proof}
Let first $u \in C^1_c(\mathbb{R}^d)$. Since $V$ is bounded from below and $C_c^\infty(\mathbb{R}^d)$ is dense in $H^{1,2}(\mathbb{R}^d,\mathrm{d}x)$, it follows that the Poincar\'e inequality in (A2) even holds true for all elements from $H^{1,2}(\mathbb{R}^d,\mathrm{d}x)$, hence is satisfied for $u$. Following identically the computation in the proof of Lemma 6 in \cite{DMS10}, Condition (A2) (applied to $u$) together with (A3) and integration by parts imply the inequality
\begin{align} \label{Inequality_1_Lm1}
\| \nabla u \|^2 &\geq   \left( \frac{1- 2 c_3}{4} \right) \|u \, \nabla V \|^2 - \frac{c_2}{2} \frac{1}{\Lambda_{1}} \| \nabla u \|^2 - \frac{c_2}{2} \left(u,1\right)_{L^2(e^{-V}\mathrm{d}x)}^2.
\end{align}
Now let $u$ be of the form $u=f - \left(f,1\right)_{L^2(e^{-V}\mathrm{d}x)}$ for some $f \in C^1_c(\mathbb{R}^d)$. Let $(\varphi_n)_{ n \in \mathbb{N}}$ be a sequence of cutoff functions from $C_c^\infty(\mathbb{R}^d)$ as in the proof of Lemma \ref{Lm_verification_D5_spherical_velocity_Langevin}. Recall that $0 \leq \varphi_n \leq 1$, $n \in \mathbb{N}$, and there exists $C< \infty$ such that
\begin{align*}
\sup_{x \in \mathbb{R}^d} \left| \nabla \varphi_n(x) \right| \leq \frac{1}{n} \, C ~\mbox{ and }~\varphi_n \rightarrow 1 \mbox{ pointwise as } n \rightarrow \infty.
\end{align*}
Define $u_n= f - \varphi_n \left(f,1\right)_{L^2(e^{-V}\mathrm{d}x)}$ for all $n \in \mathbb{N}$. By Lebesgue's dominated convergence we get 
\begin{align*}
\lim_{n \rightarrow \infty} u_n =u,~\lim_{n \rightarrow \infty} \nabla u_n = \nabla u \,\mbox{ and } \,\lim_{n \rightarrow \infty} u_n \nabla V =u \nabla V
\end{align*}
with convergence in ${L^2(e^{-V}\mathrm{d}x)}$ in each case.  And since Inequality \eqref{Inequality_1_Lm1} holds for all $u_n$, $n \in \mathbb{N}$, the claim follows. 
\end{proof}

\begin{Lm} \label{Lm_Appendix_2}
Assume that (A1)-(A5) holds. Then $W \,\nabla V \in L^2(e^{-V}\mathrm{d}x)$ and there exists $\Lambda_{3} \in (0,\infty)$ such that for all $u \in \mathcal{X}$ we have
\begin{align*}
\| W \, \nabla u \|^2 \geq \Lambda_{3} \, \| W u \, \nabla V \|^2.
\end{align*}
\end{Lm}

\begin{proof}
Let first $u \in C^\infty_c(\mathbb{R}^d)$. Then $Wu \in C^1_c(\mathbb{R}^d)$. First, Lemma \ref{Lm_Appendix_1} applied to $Wu- \left(Wu,1 \right)_{L^2(e^{-V}\mathrm{d}x)}$ gives
\begin{align*}
\| \nabla \left( Wu \right) \|^2 \geq \Lambda_{2} \, \left\| \left(Wu - \left(Wu,1 \right)_{L^2(e^{-V}\mathrm{d}x)} \right) \nabla V \, \right\|^2.
\end{align*}
By performing the computations as in the proof of \cite[Lem.~7]{DMS10} and using Conditions (A2) (applied to $u$) and (A5) one shows that
\begin{align} \label{eq_1_Lm_Appendix_2}
\frac{\Lambda_{2}}{2} \, \| W u\, \nabla V \|^2 \leq & \left( 2 + \frac {2 \Lambda_{2}}{\Lambda_{1}} \|W\|^2 \, \| \nabla V \|^2 \right) \| W \, \nabla u \|^2 + 4 \, c_4^2 \, \left\|  W \, u \right\|^2 \nonumber \\&+ {2 \Lambda_{2}} \, \|W\|^2 \, \| \nabla V \|^2 \, \left(u,1\right)^2_{L^2(e^{-V}\mathrm{d}x)}.
\end{align} 
Now let $(\varphi_n)_{n \in \mathbb{N}}$ and $C < \infty$ be as in the proof of Lemma \ref{Lm_Appendix_1}. Define $\psi_n=\varphi_{2^n}$ for all $n \in \mathbb{N}$. Then note that $\psi_n \uparrow 1$ pointwise as $n \rightarrow \infty$. Clearly, we have
\begin{align} \label{eq_2_Lm_Appendix_2}
\lim_{n \rightarrow \infty} W \psi_n =W ,~ \lim_{n \rightarrow \infty}W \, \nabla \psi_n  = 0 \,\mbox{ in }L^2(e^{-V}\mathrm{d}x)
\end{align}
since $W \in L^2(e^{-V}\mathrm{d}x)$, $ \psi_n \uparrow 1$ as $n \rightarrow \infty$ and $\left|\nabla \psi_n \right| \leq \frac{1}{2^n} \,C$ for all $n \in \mathbb{N}$. Inserting $\psi_n$, $n \in \mathbb{N}$, instead of $u$ in \eqref{eq_1_Lm_Appendix_2} and using monotone convergence on the left hand side we obtain $\int_{\mathbb{R}^d} W^2 \left| \nabla V \right|^2 \, \mathrm{d} e^{-V}\mathrm{d}x < \infty$. Now let $u$ be of the form $u=f - \left(f,1\right)_{L^2(e^{-V}\mathrm{d}x)}$ for some $f \in C^\infty_c(\mathbb{R}^d)$ and choose an approximating sequence $(u_n)_{n \in \mathbb{N}}$ for $u$ as in the proof of Lemma \ref{Lm_Appendix_1}. Hence, using also that $W \left| \nabla V \right|  \in L^2(e^{-V}\mathrm{d}x)$, we get
\begin{align*}
\lim_{n \rightarrow \infty} W \, u_n = W \,u ,~ \lim_{n \rightarrow \infty} W \, \nabla u_n  = W \nabla u,~\lim_{n \rightarrow \infty} W u_n\, \nabla V =W u \, \nabla V
\end{align*}
with convergence in $L^2(e^{-V}\mathrm{d}x)$. Consequently, \eqref{eq_1_Lm_Appendix_2} reduces for all those $u$ to
\begin{align*} 
\frac{\Lambda_{2}}{2} \, \| W u\, \nabla V \|^2 \leq & \left( 2 + \frac {2 \Lambda_{2}}{\Lambda_{1}} \|W\|^2 \, \| \nabla V \|^2 \right) \| W \, \nabla u \|^2 + 4 \, c_4^2 \, \left\| W \, u \right\|^2. \nonumber 
\end{align*} 
Finally, by (A2) and Lemma \ref{Lm_Appendix_1} the claim follows since
\begin{align*}
\left\| W \, u \right\|^2 = \left\| u \right\|^2 +  \left\| u \nabla V \right\|^2 \leq \left(\frac{1}{\Lambda_{1}} + \frac{1}{\Lambda_{2}} \right) \left\|  \nabla u \right\|^2 \leq \frac{\Lambda_{2} + \Lambda_{1}}{\Lambda_{2} \Lambda_{1}} \left\|  W \, \nabla u \right\|^2.
\end{align*}
\end{proof}

The upcoming lemma can be proven by copying the proof of Lemma 8 of \cite{DMS10} in case a solution $u \in \mathcal{X}$ to \eqref{eq_Elliptic_Equation} is considered. It makes mainly use of Lemma \ref{Lm_Appendix_1}, Lemma \ref{Lm_Appendix_2}, (A2) and (A5). 

\begin{Lm} \label{Lm_Appendix_3} 
Assume (A1)-(A5). Let $u \in \mathcal{X}$ and choose $g$ as in \eqref{eq_Elliptic_Equation} accordingly. Then there exists some constant $c_5 < \infty$, independent of $u$ and $g$, such that 
\begin{align*}
\left\| W \, u \right\|^2 + \| W\, \nabla u \|^2 \leq c_5 \,\|g\|^2.
\end{align*}
\end{Lm}

Finally, the upcoming Proposition gives the desired regularity estimate. For the proof, see the proof of \cite[Prop.~5]{DMS10}. The latter requires especially the use of Lemma \ref{Lm_Appendix_3}. 

\begin{Pp} \label{Elliptic_regularity_Prop}
Assume (A1)-(A5). Let $u \in \mathcal{X}$ and choose $g$ as in \eqref{eq_Elliptic_Equation} accordingly. Then there exists some constant $c_6 < \infty$, independent of $u$ and $g$, such that 
\begin{align*}
\| \nabla^2 u \| \leq c_6 \, \| g\|.
\end{align*}
\end{Pp}

Finally, here are  some sufficient conditions implying (A1)-(A5), used already e.g.~in \cite{DKMS11}.

\begin{Lm}
Let the potential $V$ satisfy Assumptions (A1) and (A2). Furthermore, assume that there exists a constant $c < \infty$ such that
\begin{align} \label{Villanis_condition}
\left| \nabla^2 V (x) \right| \leq c \left( 1+ \left| \nabla V(x) \right|\right) \mbox{ for all } x \in \mathbb{R}^d.
\end{align}
Then $V$ fulfills (A1)-(A5). We remark that \eqref{Villanis_condition} is some potential growing condition introduced by Villani, see \cite{Vil09}.
\end{Lm}

\begin{proof}
We refer to \cite{DKMS11}, \cite{DMS10} and \cite{Vil09}. Indeed, (A3) follows by \eqref{Villanis_condition} and the Young inequality, see \cite{DKMS11}. (A4) mainly follows from  \eqref{Villanis_condition}, see e.g.~\cite[Lem.~A18]{Vil09}.  Finally, (A5) can easily be proven using again \eqref{Villanis_condition}. 
\end{proof}

\section*{Acknowledgement} This work has been supported by Bundesministerium f"{u}r Bildung und Forschung, Schwerpunkt \glqq Mathematik f"{u}r Innovationen in Industrie and Dienstleistungen\grqq , Verbundprojekt  ProFil, $03$MS$606$. Furthermore, the second named author is grateful to Benedict Baur, Martin Kolb and Axel Klar for discussions and helpful comments. Finally, the authors thank Dietmar Hietel and Raimund Wegener for their virtual nonwoven webs provided in Section \ref{Hypocoercivity_section_Spherical_velocity_Langevin}.

\end{document}